\documentclass[12pt,reqno,a4paper]{amsart}
\usepackage{amsmath}
\usepackage{amsfonts}
\usepackage{amssymb}
\usepackage[all,cmtip]{xy}           
\usepackage{bm}
\usepackage{bbm}
\usepackage{bbding}
\usepackage{txfonts}
\usepackage{amscd}
\usepackage{xspace}
\usepackage[shortlabels]{enumitem}
\usepackage{ifpdf}

\usepackage{subfig}
\ifpdf
\usepackage[colorlinks,final,hyperindex]{hyperref}
\else
\usepackage[colorlinks,final,hyperindex,hypertex]{hyperref}
\fi
\usepackage{tikz}
\usepackage[active]{srcltx}

\usepackage{tikz-cd}
\usepackage{tikz}

\makeatletter

\newtheorem{thm}{Theorem}[section]
\newtheorem{lem}[thm]{Lemma}

\newtheorem{cor}[thm]{Corollary}
\newtheorem{pro}[thm]{Proposition}
\theoremstyle{definition}
\newtheorem{ex}[thm]{Example}
\newtheorem{rmk}[thm]{Remark}
\newtheorem{defi}[thm]{Definition}

\newcommand{\nc}{\newcommand}
\newcommand{\delete}[1]{}

	\nc{\mlabel}[1]{\label{#1}}  
	\nc{\mcite}[1]{\cite{#1}}  
	\nc{\mref}[1]{\ref{#1}}  
	\nc{\meqref}[1]{\eqref{#1}}  
	\nc{\mbibitem}[1]{\bibitem{#1}} 

\delete{
	\nc{\mcite}[1]{\cite{#1}{{\bf{{\ }(#1)}}}}  
	\nc{\mref}[1]{\ref{#1}{{\bf{{\ }(#1)}}}}  
	\nc{\meqref}[1]{\eqref{#1}{{\bf{{\ }(#1)}}}}  
	\nc{\mbibitem}[1]{\bibitem[\bf #1]{#1}} 
}

\newcommand {\emptycomment}[1]{}

\nc{\oprn}{\theta}

\nc{\ENL}{{\rm ENL }}
\nc{\ENLE}{{\rm ENL}}
\nc{\ENE}{{\rm EN}}
\nc{\EN}{{\rm EN }}
\nc{\NL}{{\rm NL }}

\nc{\calo}{\mathcal{O}}
\nc{\oop}{$\mathcal{O}$-operator\xspace}
\nc{\oops}{$\mathcal{O}$-operators\xspace}
\nc{\mrho}{{\bm{\varrho}}}
\nc{\bfk}{\mathbf{K}}
\nc{\invlim}{\displaystyle{\lim_{\longleftarrow}}\,}
\nc{\ot}{\otimes}

\nc{\eval}[1]{\Big|_{#1}}

\newcommand{\add}{\frka\frkd}

\newcommand{\be }{\begin{equation}}
	\newcommand{\ee }{\end{equation}}

\def\lcf{\lbrack\! \lbrack}
\def\rcf{\rbrack\! \rbrack}

\newcommand{\g}{\mathfrak g}
\newcommand{\h}{\mathfrak h}

\nc{\RR}{\mathbb{R}}

\nc{\CC}{\mathbb{C}}

\newcommand{\huaE}{\mathcal{E}}

\newcommand{\huaG}{\mathcal{G}}

\newcommand{\huaI}{\mathcal{I}}

\newcommand{\frka}{\mathfrak a}

\newcommand{\frkd}{\mathfrak d}

\newcommand{\frkE}{\mathfrak E}

\newcommand{\Id}{{\rm{Id}}}

\newcommand{\br}[1]{   [ \cdot,    \cdot  ]   }

\newcommand{\Hom}{\mathrm{Hom}}

\newcommand{\gl}{\mathfrak {gl}}

\newcommand{\ad}{\mathrm{ad}}

\nc{\CV}{\mathbf{C}}

\begin{document}

	\title[ENL algebras]{Equivariant Nijenhuis Lie algebras:\\
		extensions to classical Lie-theoretic structures}

	\author{Shuai Hou}
	\address{Department of Mathematics, Jilin University, Changchun 130012, Jilin, China}
	\email{hshuaisun@jlu.edu.cn}
	
\author{Zohreh Ravanpak}

\address{
	School of Physical and Mathematical Sciences,
	Nanyang Technological University,
	Singapore
	\\
	and
	\\
	Max Planck Institute for Mathematics,
	Bonn, Germany
}
\email{zohreh.ravanpak@ntu.edu.sg; ravanpak@mpim-bonn.mpg.de}

	\author{Yunhe Sheng}
	\address{Department of Mathematics, Jilin University, Changchun 130012, Jilin, China}
	\email{shengyh@jlu.edu.cn}
	
	
	\begin{abstract}
		We develop a theory of \emph{equivariant Nijenhuis Lie algebras} (\ENL algebras), namely Lie algebras equipped with Nijenhuis operators satisfying an equivariance condition with respect to the adjoint representation. This compatibility condition allows classical Lie bialgebra constructions to extend naturally to the operator-equipped setting. Within this framework, we define \ENL bialgebras and establish the associated notions of matched pairs, Manin triples, and Drinfel'd doubles. We show that coboundary \ENL bialgebras are characterized by $\EN$ $r$-matrices satisfying an equivariant classical Yang--Baxter equation. We further introduce $\EN$-relative Rota--Baxter operators and prove that they provide an operator-theoretic realization of such $r$-matrices, leading to descendant \ENL algebras and to solutions of the classical Yang--Baxter equation on semidirect \ENL algebras. In the quadratic case, this construction recovers Rota--Baxter operators of weight zero.
		Finally, we extend the \EN framework to pre-Lie algebras and show that pre-\ENL algebras naturally induce associated ENL structures.

	\end{abstract}

	
	\keywords{Equivariant Nijenhuis Lie algebra, Equivariant Nijenhuis Lie bialgebra,  Matched pair, classical Yang--Baxter equation, Rota--Baxter operator\\
		\quad  2020 \emph{Mathematics Subject Classification.}   17B62, 17B10, 17B38.}
	
	\maketitle
	
	\tableofcontents

	\setcounter{section}{-1}
	\section{Introduction}

	Nijenhuis operators on Lie algebras provide a natural mechanism for encoding compatibility between Lie brackets and for generating new brackets via iterated deformations. In this way, they organize families of compatible Lie algebra structures and play a central role in hierarchy generation and in the formation of doubled or interconnected Lie-theoretic objects. When combined with classical constructions such as matched pairs, Manin triples, and Drinfel'd doubles, Nijenhuis operators raise structural questions concerning stability under duality, doubling, and bialgebra-type extensions. Understanding these compatibility mechanisms is essential for a systematic theory of Lie algebras equipped with additional operator structure.
	
	\smallskip
	
	In classical Lie theory, finite-dimensional Lie bialgebras admit equivalent descriptions in terms of matched pairs or Manin triples. A matched pair encodes compatible mutual actions of a Lie algebra and its dual such that the associated bicrossed product carries a natural Lie bracket, while a Manin triple consists of a quadratic Lie algebra decomposed into complementary isotropic Lie subalgebras. This equivalence is realized by the Drinfel'd double construction, which identifies Lie bialgebras with quadratic Lie algebras containing both a Lie algebra and its dual as isotropic subalgebras. Within this framework, classical $r$-matrices arise from Lagrangian splittings of the Drinfel'd double, equivalently from choices of complementary isotropic subalgebras in a Manin triple. These $r$-matrices satisfy the classical Yang--Baxter equation and provide a fundamental mechanism underlying many integrable systems \cite{BD,D,KM,RS}. In this sense, the Drinfel'd double furnishes a unifying quadratic Lie algebra setting in which Lie bialgebras, Manin triples, and classical $r$-matrices emerge from different choices of complementary isotropic subalgebras.

\smallskip

Motivated by this compatibility paradigm, we develop an analogous correspondence for Lie algebras equipped with Nijenhuis-type operators. Our basic object is an \emph{equivariant Nijenhuis Lie algebra} (\ENL algebra), consisting of a Lie algebra $\g$ together with a Nijenhuis operator $E$ satisfying a natural equivariance condition with respect to the adjoint representation. Within this setting, we introduce matched pairs of \ENL algebras and show that their bicrossed products inherit canonical \ENL structures. Conversely, we define Manin triples of \ENL algebras and prove that these notions are equivalent. More precisely, for a Lie algebra $\g$ endowed with an equivariant Nijenhuis operator $E$, the pair $(\g,E)$ and its dual $(\g^*,E^*)$ form a matched pair if and only if the direct sum $(\g\oplus\g^*,\,E\oplus E^*)$ carries a quadratic \ENL algebra structure for which $(\g,E)$ and $(\g^*,E^*)$ are isotropic \ENL subalgebras.

Rota--Baxter operators complement this framework by encoding compatible splittings of the Lie bracket and by providing operator-theoretic realizations of factorizable and coboundary Lie bialgebras \cite{Bai,BGN2010,Ku,STS}. We extend this correspondence to the \ENL setting by introducing \ENL\ Rota--Baxter algebras, which play a role parallel to their classical counterparts. Furthermore, we define \ENE-relative Rota--Baxter algebras and show that they give rise to classical $\EN$ $r$-matrices, namely solutions of the classical Yang--Baxter equation adapted to the equivariant Nijenhuis context. These structures naturally lift to coboundary \ENL bialgebras and their associated doubles.

Finally, we extend the \ENL framework to \emph{pre-Lie algebras}, which arise naturally in affine geometry and deformation theory. A pre-Lie algebra refines a Lie algebra structure through a nonassociative product whose commutator recovers the underlying Lie bracket. We introduce \emph{pre-\ENL algebras}, and prove that the equivariance condition is compatible with the \ENL constructions developed above, thereby allowing pre-\ENL algebras to integrate naturally into the same compatibility framework. This refinement establishes further connections with relative Rota--Baxter operators and yields additional constructions of classical $r$-matrices.

		\smallskip
		It is important to distinguish \ENL bialgebras from the previously introduced \NL bialgebras \cite{Zohreh}, which may be regarded as algebraic counterparts of Poisson--Nijenhuis structures on Lie bialgebras. Such \NL bialgebras provide an algebraic framework for certain bi-Hamiltonian systems, including those studied in \cite{BaMaRa}. In these constructions, \NL bialgebras that act as hierarchy generators typically impose an asymmetric equivariance condition. This reflects the geometric situation on a Poisson--Nijenhuis manifold $(M,\Pi,N)$, where equivariance is required either for the Nijenhuis tensor $N$ or for its transpose $N^*$ on $\mathsf{T}^*M$, but not simultaneously for both.
			While this asymmetry is natural from a geometric viewpoint, it obstructs classical double constructions. In general, the block-diagonal operator $N\oplus N^*$ fails to define a Nijenhuis operator on the double $\g\bowtie\g^*$, and consequently the classical correspondence between matched pairs and Manin triples may break down in the \NL setting. The \ENL framework resolves this issue by requiring equivariance for both $E$ and $E^*$, thereby ensuring that the double carries a quadratic \ENL algebra structure and restoring the matched-pair/Manin-triple correspondence.

\smallskip

Previous work has explored Nijenhuis-type structures in both geometric and algebraic settings. On the geometric side, invariant Poisson and Poisson–quasi-Nijenhuis structures on Lie groups have been investigated in \cite{HaZo,RaReHa}, while Poisson–Nijenhuis manifolds provide a fundamental framework for bi-Hamiltonian systems and integrable hierarchies \cite{DamianouFernandes,KosmannPN,Magri,MagriMorosiAnn}. From an algebraic perspective, Nijenhuis operators on Lie algebras and Poisson pencils have long been used to construct compatible brackets, recursion operators, and integrable systems \cite{Panasyuk2006}.  In addition, \NL bialgebras were introduced in \cite{Zohreh} as algebraic counterparts of Poisson-Nijenhuis structures, and classical $r$-matrix methods further relate Lie bialgebra doubles to integrability phenomena \cite{FokasFuchssteiner,Vicedo2015}. These developments provide the broader context and motivation for the equivariant Nijenhuis framework developed in the present work.
\smallskip

It is worth clarifying the position of the present work relative to recent studies on Nijenhuis Lie bialgebras. In \cite{Das-1}, such structures are developed from a cohomological and deformation-theoretic viewpoint, based on admissibility conditions between generally independent Nijenhuis operators on the algebraic and coalgebraic sides. In \cite{Li-Ma}, the same definition is adopted, with emphasis on the coboundary case and explicit constructions arising from solutions of the classical Yang--Baxter equation.
By contrast, the present paper addresses a complementary structural problem. By imposing equivariance and self-duality conditions on the Nijenhuis operator—compatible with duality and quadratic structures—we obtain a setting in which the classical Lie-theoretic correspondences between matched pairs, Manin triples, Drinfel'd doubles, and classical $r$-matrices persist within the category of operator-equipped Lie algebras. Since these structures satisfy symmetric equivariance conditions on both the operators $E$ and $E^*$, we refer to them as \ENL bialgebras.

	\smallskip
	
A central theme of this work is the stability of bialgebraic constructions under the introduction of compatible operator data. Whether one begins with a Lie bialgebra and imposes an equivariant Nijenhuis structure, or instead develops a bialgebra theory directly within the category of operator-equipped Lie algebras, the resulting structures coincide up to natural equivalence. In this sense, the Nijenhuis operator is not a supplementary decoration but an intrinsic component of the structure, governing duality, the existence of doubles, and the organization of compatible hierarchies. Rota--Baxter operators and pre-Lie algebras serve throughout as concrete mechanisms that both illustrate and generate these constructions.

	\medskip
	\noindent\textbf{Structure of the paper.}
	Section~\ref{sec:l} develops the basic theory of \ENL algebras, including their representations, duals, and quadratic
structures.  These results provide the algebraic framework for the matched
pairs and Manin triples studied later.
		Section~\ref{sec:mp} develops the theory of matched pairs and Manin triples
	of \ENL algebras and shows that, as in the classical Lie setting, these two
	notions provide equivalent descriptions of \ENL doubles.
In Section~\ref{sec:bia}, we develop the structure theory of \ENL bialgebras,
with particular emphasis on their realization via matched pairs and
Drinfel'd doubles in the \ENL setting. Section~\ref{sec:bia2} studies equivariant Nijenhuis structures on
quadratic Rota-Baxter Lie algebras, introduces the notion of \ENLE-{\rm RB}
algebras, and shows that the \ENL framework is stable under the
Rota-Baxter descendant construction. In Section~\ref{sec:enr}, we study the classical Yang-Baxter equation in the
equivariant Nijenhuis setting, introduce \EN $r$-matrices, and show that they
generate coboundary \ENL  bialgebras. Section~\ref{sec.6} is devoted to \ENE-relative Rota-Baxter operators and the \ENL structures they induce, including descendant \ENL algebras and associated bialgebra-type constructions.
	Finally, Section~\ref{sec.7} extends the theory to pre-Lie algebras by
	introducing equivariant Nijenhuis operators on pre-Lie algebras, and by defining the
	resulting pre-\ENL algebras as refinements of \ENL algebras.

\medskip

Throughout this paper, we work over an algebraically closed field of characteristic zero, and all vector spaces are assumed to be finite-dimensional.

\section{Equivariant Nijenhuis Lie (\ENLE) algebras}\label{sec:l}
In this section, we introduce a rigid subclass of Nijenhuis operators on Lie
algebras, namely \emph{equivariant} Nijenhuis operators.  We begin by recalling
the standard notion of Nijenhuis torsion.  For a linear operator
$N:\g\to\g$ on a Lie algebra $(\g,[\cdot,\cdot]_{\g})$, its Nijenhuis torsion is
the bilinear map $\mathcal T_N:\g\times\g\to\g$ defined by
\begin{equation}\label{Nijenhuis-operator}
	\mathcal{T}_N(x,y)
	:=
	[Nx,Ny]_{\g}
	-
	N\big([Nx,y]_{\g}+[x,Ny]_{\g}-N[x,y]_{\g}\big),
	\qquad \forall x,y\in\g.
\end{equation}
The operator $N$ is called a \emph{Nijenhuis operator} if $\mathcal T_N\equiv 0$.
In this case the \emph{$N$-deformed bracket}
\[
[x,y]_N := [Nx,y]_{\g} + [x,Ny]_{\g} - N[x,y]_{\g}
\]
is also a Lie bracket on $\g$, and $N$ becomes a Lie algebra morphism
from $(\g,[\cdot,\cdot]_N)$ to $(\g,[\cdot,\cdot]_{\g})$.
Accordingly, the deformed adjoint representation
${\ad}^N:\g\to\gl(\g)$ of the Lie algebra $(\g,[\cdot,\cdot]_N)$ is given by
\begin{equation}\label{ad-n}
	{\ad}^{N}_{x}
	=
	\underline{\ad}_{\,x}+\ad_{Nx},
	\qquad  x\in\g,
\end{equation}
where $\underline{\ad}_{\,x}:=[\ad_x,N]=\ad_xN-N\ad_x$ (cf.~\cite{Zohreh}).

In this notation, the Nijenhuis condition can be written compactly as
\begin{equation}\label{N-ad}
	N\circ \underline{\ad}_{\,x}=\underline{\ad}_{\,Nx},
	\qquad  x\in\g,
\end{equation}
expressing the compatibility between $N$ and the adjoint action.
Iterating this compatibility produces hierarchies of deformed brackets, which
is one reason Nijenhuis operators play a central role in deformation theory and
integrable systems; see \cite{KoMa}.
A Lie algebra equipped with a Nijenhuis operator is called a
\emph{Nijenhuis Lie algebra (\NL algebra)} and is denoted by
$(\g,[\cdot,\cdot]_{\g},N)$.

We now single out a rigid subclass of Nijenhuis operators by imposing an
additional equivariance condition with respect to the adjoint action
$\ad_x(y)=[x,y]_{\g}$.

\begin{defi}
	Let $(\g,[\cdot,\cdot]_{\g})$ be a Lie algebra.
	A linear operator $E:\g\to\g$ is called an
	\emph{equivariant Nijenhuis operator} if $\underline{\ad}_{\,x}=0$, i.e. $E$ commutes with the adjoint
	representation:
	\begin{equation}\label{equivariant-Nijenhuis}
		E[x,y]_{\g} = [x,Ey]_{\g},
		\qquad\text{equivalently}\qquad
		E\circ\ad_x=\ad_x\circ E,
		\quad \forall\,x,y\in\g.
	\end{equation}
	A Lie algebra equipped with such an operator is called an
	\emph{equivariant Nijenhuis Lie algebra} (\ENL algebra) and is denoted by
	$(\g,[\cdot,\cdot]_{\g},E)$, or simply by $(\g,E)$.
\end{defi}

\smallskip

If, in addition, $E^2=\Id$, then $E$ defines a \emph{product structure} on $\g$;
see \cite{Andrada}.
Note that~\eqref{equivariant-Nijenhuis} is symmetric in the arguments of the
Lie bracket. Equivalently, it can be written as
\[
E[x,y]_\g = [Ex,y]_\g,\qquad \forall x,y\in\g.
\]

\begin{lem}\label{lem:equiv-implies-nij}
	Let $(\g,[\cdot,\cdot]_\g,E)$ be an \ENL algebra.
	Then $E$ is a Nijenhuis operator.
\end{lem}

\begin{proof}
	Since $\underline{\ad}_{\,x}=0,$ it follows that \eqref{N-ad} holds.
\end{proof}

 In this case, the $E$-deformed bracket associated with an
\ENL algebra $(\g,[\cdot,\cdot]_{\g},E)$ takes the particularly simple form
\begin{equation}\label{E-deformed-bracket}
	[x,y]_{E} := [Ex,y]_{\g},
	\qquad \forall\,x,y\in\g.
\end{equation}
The simplification of the $E$-deformed bracket has important structural
consequences.

\begin{pro}\label{induce-Lie-algebra}
	Let $(\g,[\cdot,\cdot]_{\g},E)$ be an \ENL algebra.
	Then $(\g,[\cdot,\cdot]_{E},E)$ is an \ENL algebra, where
	$[x,y]_{E}:=[Ex,y]_{\g}$.
\end{pro}

\begin{proof}
	We have
	\[
	E[x,y]_{E}=E[Ex,y]_{\g}=[E^{2}x,y]_{\g}=[Ex,y]_{E},
	\]
which implies that $E$ is an equivariant Nijenhuis operators on $[\cdot,\cdot]_{E}$.
\end{proof}

\begin{cor}
	Every \ENL algebra $(\g,[\cdot,\cdot]_{\g},E)$ gives rise to a hierarchy of
	\ENL algebras
	\[
	(\g,[\cdot,\cdot]_{E^{k}},E^{k}),
	\qquad k\in\mathbb Z_{>0}.
	\]
\end{cor}

\begin{rmk}\label{equiv-ad}
	Let $(\g,[\cdot,\cdot]_{\g},E)$ be an \ENL algebra.
	Since
	$\underline{\ad}_{\,x}=[\ad_x,E]=0$ for all $x\in\g$,
	  the deformed adjoint representation associated with the
	$E$-deformed bracket $[\cdot,\cdot]_E$ reduces to
	\begin{equation}\label{adn-equiv}
		\ad^{E}_{x}
		=\ad_{Ex}
		=\ad_x\circ E
		=E\circ\ad_x,
		\qquad x\in\g.
	\end{equation}
\end{rmk}

\begin{pro}\label{E-inverse}
	Let $E:\g\to\g$ be an invertible equivariant Nijenhuis operator.
	Then its inverse $E^{-1}:\g\to\g$ is again an equivariant Nijenhuis
	operator.
\end{pro}

\begin{proof}
	For all $x,y\in\g$, using the equivariance of $E$ and its invertibility,
	we compute
	\[
	E^{-1}[x,y]_{\g}
	=
	E^{-1}\big([E(E^{-1}x),E(E^{-1}y)]_{\g}\big)
	=
	[x,E^{-1}y]_{\g}\,.
	\]
Thus, $E^{-1}$ commutes with the adjoint action, and hence is an
equivariant Nijenhuis operator.
\end{proof}
Another class of endomorphisms that interacts naturally with operator-based
deformations of Lie brackets is given by \emph{averaging operators}
\cite{Aguiar}.  Although averaging operators are not, in general, Nijenhuis
torsion-free, they provide a classical mechanism for producing new algebraic
brackets from a Lie algebra via a linear endomorphism.  We recall this notion
because, in the invertible case, it coincides with the \ENL condition, and it
also serves as a useful point of comparison for the double-type constructions
that appear later in the \ENL setting..
\smallskip

A linear map $P:\g\to\g$ is called a \emph{left averaging operator} if it satisfies
\begin{equation}\label{averaging-operator}
	[Px,Py]_{\g} = P([Px,y]_{\g}),
	\qquad \forall\,x,y\in\g.
\end{equation}
Since no torsion-free condition is imposed on $P$, the resulting deformation is
not of Nijenhuis type.  Instead, one considers the bilinear operation
\[
[x,y]_{P} := [Px,y]_{\g},
\]
which, under condition \eqref{averaging-operator}, equips $\g$ with a natural
left Leibniz algebra structure (typically non-skew-symmetric), denoted by
$\g_{P}$.

\begin{pro}\label{ENL-averaging}
	Let $(\g,[\cdot,\cdot]_{\g},E)$ be an \ENL algebra.
	Then $E$ is an averaging operator.
	Conversely, if an averaging operator $E:\g\to\g$ is invertible, then it defines an \ENL algebra
	$(\g,[\cdot,\cdot]_{\g},E)$.
\end{pro}

\begin{proof}
	If $(\g,[\cdot,\cdot]_{\g},E)$ is an \ENL algebra, then equivariance gives	$
	[Ex,Ey]_{\g} = E[Ex,y]_{\g},
	$
	so $E$ satisfies \eqref{averaging-operator}.

	Conversely, let $E$ be an invertible averaging operator and set
	$u=Ex$, $v=Ey$.  Then by \eqref{averaging-operator},
	\[
	E^{-1}[u,v]_{\g}
	=
	E^{-1}[Ex,Ey]_{\g}
	=
	[Ex,y]_{\g}
	=
	[u,E^{-1}v]_{\g},
	\]
	which shows that $E^{-1}$ commutes with the adjoint action.
	By Proposition~\ref{E-inverse}, $E$ is therefore an equivariant Nijenhuis
	operator, and hence $(\g,[\cdot,\cdot]_{\g},E)$ is an \ENL algebra.
\end{proof}
\begin{defi}
	Let $(\g,[\cdot,\cdot]_{\g},E)$ and $(\g',[\cdot,\cdot]_{\g'},E')$ be two \ENL
	algebras.
	An \emph{\ENE-homomorphism} from $(\g,E)$ to $(\g',E')$ is a Lie algebra
	homomorphism $\phi:\g\to\g'$ satisfying
	\[
	\phi\circ E = E'\circ \phi.
	\]
	If $\phi$ is bijective, it is called an \emph{\ENE-isomorphism}.
\end{defi}

\smallskip

We next recall the notion of a representation compatible with a Nijenhuis
operator, which will serve as a stepping stone toward the equivariant case.

\begin{defi}
	Let $(\g,[\cdot,\cdot]_{\g},N)$ be a Nijenhuis Lie algebra.
	An \emph{$N$-representation} of $\g$ consists of a representation
	$(W;\rho)$ of the Lie algebra $(\g,[\cdot,\cdot]_{\g})$ together with a
	linear map $T:W\to W$ such that for all $x\in\g$ and $u\in W$ the
	compatibility condition
	\begin{equation}\label{E-rep}
		T^{2}(\rho(x)u)
		+ \rho(Nx)(Tu)
		- T(\rho(Nx)u)
		- T(\rho(x)(Tu))
		= 0
	\end{equation}
	is satisfied.
	In this case we say that $T$ is \emph{compatible} with $N$.
	The identity~\eqref{E-rep} can be interpreted as a balance condition
	between the actions of $T$, $\rho(x)$, and $\rho(Nx)$, and is illustrated
	by the diagram in Figure~\ref{d1}.
\end{defi}
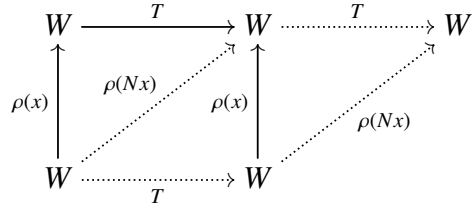
\begin{figure}[h]
	\centering
	\subfloat{
		\xymatrix@C=2cm@R=1.5cm{
			W \ar[r]^{T} & W \ar@{.>}[r]^{T} & W \\
			W \ar[u]^{\rho(x)}
			\ar@{.>}[r]_{T}
			\ar@{.>}[ur]^{\rho(Nx)}
			&
			W \ar[u]^{\rho(x)}
			\ar@{.>}[ur]_{\rho(Nx)}
	}}
	\caption{
		The compatibility condition \eqref{E-rep} for an $N$-representation.
		The identity expresses the vanishing of the signed sum of all compositions
		of the maps $T$, $\rho(x)$, and $\rho(Nx)$ along paths from the lower-left
		copy of $W$ to the upper-right one.
		The dotted arrows indicate the weaker intertwining relations associated
		with an averaging operator; see Remark~\ref{aver-com}.
	}
	\label{d1}
\end{figure}
\begin{defi}
	An \emph{\ENE-representation} of an \ENL algebra
	$(\g,[\cdot,\cdot]_{\g},E)$ consists of a representation $(W;\rho)$ of the
	Lie algebra $(\g,[\cdot,\cdot]_{\g})$ together with a linear map
	$T:W\to W$ such that
	\begin{equation}\label{representation-equivariant-Liealg}
		T\big(\rho(x)u\big)
		=
		\rho(Ex)\,u
		=
		\rho(x)\big(Tu\big),
		\qquad \forall\,x\in\g,\; u\in W.
	\end{equation}
	In this case we say that $T$ is \emph{equivariant} with respect to $E$.
	Equivalently, all subdiagrams of Figure~\ref{d1}, as well as the full
	diagram with $N:=E$, are commutative.
	We denote such a representation by $(W;T,\rho)$.
\end{defi}

\begin{rmk}\label{aver-com}
	An averaging operator $P:\g\to\g$ is, in general, not Nijenhuis
	torsion-free.
	Given a representation $(W;\rho)$ of the Lie algebra $\g$ and a linear map
	$T:W\to W$, a natural compatibility condition between $\rho$ and $P$ is
	\[
	T(\rho(Px)u) = \rho(Px)(Tu),
	\qquad \forall\,x\in\g,\; u\in W.
	\]
	This condition is equivalent to the commutativity of the subdiagram
	indicated by the dotted arrows in Figure~\ref{d1}.
\end{rmk}

\begin{cor}\label{adjoint-representation-der}
	Let $(\g,[\cdot,\cdot]_{\g},E)$ be an \ENL algebra.
	Then the triple $(\g;E,\ad)$ defines an \ENE-representation,
	called the \emph{adjoint \ENE-representation}.
\end{cor}

\begin{pro}\label{semidirect-ENL}
	Let $(W;T,\rho)$ be an \ENE-representation of an \ENL algebra
	$(\g,[\cdot,\cdot]_{\g},E)$.
	Then the direct sum $\g\oplus W$, equipped with the semidirect product
	bracket
	\[
	[x+u,\,y+v]_{\ltimes}
	=
	[x,y]_{\g}+\rho(x)v-\rho(y)u,
	\qquad \forall x,y\in\g,\; u,v\in W,
	\]
	and the linear map
	\[
	\widehat{E}(x+u):=Ex+Tu,
	\]
	forms an \ENL algebra.
\end{pro}

\begin{proof}
	The bracket $[\cdot,\cdot]_{\ltimes}$ is the standard semidirect product
	bracket and hence defines a Lie algebra structure on $\g\oplus W$.
	It remains to verify that $\widehat{E}$ is equivariant with respect to
	$[\cdot,\cdot]_{\ltimes}$.
	For $x,y\in\g$ and $u,v\in W$, we compute
	\[
	\widehat{E}[x+u,\,y+v]_{\ltimes}
	=
	E[x,y]_{\g}+T\big(\rho(x)v-\rho(y)u\big),
	\]
	while, using the \ENE-representation condition
	\eqref{representation-equivariant-Liealg},
	\[
	[\widehat{E}(x+u),\,y+v]_{\ltimes}
	=
	[Ex,y]_{\g}+\rho(Ex)v-\rho(y)(Tu).
	\]
	Since $E$ is equivariant on $\g$ and $T(\rho(x)u)=\rho(Ex)u$, the two expressions
	coincide, and therefore
	\[
	\widehat{E}[x+u,\,y+v]_{\ltimes}
	=
	[\widehat{E}(x+u),\,y+v]_{\ltimes}.
	\]
	Hence $\widehat{E}$ is an equivariant Nijenhuis operator on
	$\g\oplus W$, and the conclusion follows.
\end{proof}

The semidirect \ENL algebra constructed in Proposition~\ref{semidirect-ENL}
will serve as a basic building block for matched pairs of \ENL algebras in the
next section.
\medskip

We next introduce the \emph{dual \ENE-representation} associated with an
\ENE-representation.
Let $(W;T,\rho)$ be an \ENE-representation of an \ENL algebra
$(\g,[\cdot,\cdot]_{\g},E)$.

Define the dual action $\rho^{*}:\g\to\gl(W^{*})$ by
\[
\langle \rho^{*}(x)\xi,\,u\rangle
=
-\langle \xi,\,\rho(x)u\rangle,
\qquad
\forall\,x\in\g,\; u\in W,\; \xi\in W^{*}.
\]
It is well known that $\rho^{*}$ is a representation of $\g$ on $W^{*}$.
Let $T^{*}:W^{*}\to W^{*}$ denote the linear map dual to $T$, defined by
\begin{equation}\label{dual-Reynolds-operator}
	\langle T^{*}\xi,\,u\rangle
	=
	\langle \xi,\,Tu\rangle,
	\qquad \forall\,u\in W,\; \xi\in W^{*}.
\end{equation}

\begin{pro}\label{dual-representation}
	The triple $(W^{*};T^{*},\rho^{*})$ is an \ENE-representation of the \ENL
	algebra $(\g,[\cdot,\cdot]_{\g},E)$.
	It is called the \emph{dual \ENE-representation}.
\end{pro}

\begin{proof}
	Since $\rho^{*}$ is a representation of $\g$, it suffices to verify the
	\ENE-equivariance condition.
	For $x\in\g$, $u\in W$ and $\xi\in W^{*}$, using the definitions
of $\rho^{*}$ and $T^{*}$ together with
	\eqref{representation-equivariant-Liealg}, we compute
	\[
	\big\langle T^{*}(\rho^{*}(x)\xi)-\rho^{*}(Ex)\xi,\,u\big\rangle
	=
	\langle \xi,\,-\rho(x)(Tu)+\rho(Ex)u\rangle
	=0,
	\]
	and similarly,
	\[
	\big\langle \rho^{*}(Ex)\xi-\rho^{*}(x)(T^{*}\xi),\,u\big\rangle
	=
	\langle \xi,\,-\rho(Ex)u+T(\rho(x)u)\rangle
	=0.
	\]
	Since the pairing with every $u\in W$ vanishes, both identities hold in
	$W^{*}$, and $(W^{*};T^{*},\rho^{*})$ is an \ENE-representation.
\end{proof}
Combining Corollary~\ref{adjoint-representation-der} with
Proposition~\ref{dual-representation}, we obtain the following canonical
example.

\begin{cor}\label{thm:coadjoint}
	Let $(\g,[\cdot,\cdot]_{\g},E)$ be an \ENL algebra.
	Then $(\g^{*};E^{*},\ad^{*})$ is an \ENE-representation of
	$(\g,[\cdot,\cdot]_{\g},E)$, called the \emph{coadjoint \ENE-representation}.
\end{cor}

\smallskip

We now introduce quadratic \ENL algebras, which provide the natural setting
for Manin triples in the equivariant Nijenhuis framework.
In this case, an invariant bilinear form identifies the adjoint and
coadjoint \ENE-representations.

Recall that a nondegenerate symmetric bilinear form
$S\in\otimes^{2}\g^{*}$ on a Lie algebra $\g$ is said to be
\emph{invariant} if
\begin{equation}\label{Invariant}
	S([x,y]_{\g},z) + S(y,[x,z]_{\g}) = 0,
	\qquad \forall\,x,y,z\in\g.
\end{equation}
A \emph{quadratic Lie algebra} is a pair $(\g,S)$ consisting of a Lie algebra
$\g$ and a nondegenerate symmetric invariant bilinear form $S$.

\begin{defi}\label{defi:qua}
	Let $(\g,[\cdot,\cdot]_{\g},E)$ be an \ENL algebra and
	$S\in\otimes^{2}\g^{*}$ be a nondegenerate symmetric bilinear form.
	The triple $(\g,E,S)$ is called a \emph{quadratic \ENL algebra} if
	$(\g,S)$ is a quadratic Lie algebra and the following compatibility
	condition is satisfied:
	\begin{equation}\label{Eder-manin}
		S(Ex,y) = S(x,Ey),
		\qquad \forall\,x,y\in\g.
	\end{equation}
\end{defi}

\begin{defi}
	Let $(W;T,\rho)$ and $(W';T',\rho')$ be two \ENE-representations of an
	\ENL algebra $(\g,[\cdot,\cdot]_{\g},E)$.
	A \emph{homomorphism} from $(W;T,\rho)$ to $(W';T',\rho')$ is a linear map
	$\phi:W\to W'$ such that
	\[
	\phi\circ\rho(x)=\rho'(x)\circ\phi,
	\qquad
	\phi\circ T = T'\circ\phi,
	\qquad \forall\,x\in\g.
	\]
\end{defi}

\begin{thm}\label{quadratic-Eder}
	Let $(\g,E,S)$ be a quadratic \ENL algebra.
	Then the linear map
	\[
	S^{\sharp}:\g\to\g^{*},
	\qquad
	\langle S^{\sharp}(x),y\rangle = S(x,y),
	\quad \forall\,x,y\in\g,
	\]
	is an isomorphism from the adjoint \ENE-representation $(\g;E,\ad)$ to the
	coadjoint \ENE-representation $(\g^{*};E^{*},\ad^{*})$.
\end{thm}

\begin{proof}
	The invariance of $S$ is equivalent to the intertwining relation
	\begin{equation}\label{quad-intertwine-ad}
		S^{\sharp}\circ\ad_x = \ad_x^{*}\circ S^{\sharp},
		\qquad \forall\,x\in\g.
	\end{equation}
	The compatibility condition \eqref{Eder-manin} further implies
	\[
	\langle S^{\sharp}(Ex),y\rangle
	=
	\langle S^{\sharp}(x),Ey\rangle
	=
	\langle E^{*}(S^{\sharp}(x)),y\rangle,
	\]
	and hence
	\begin{equation}\label{quad-intertwine-E}
		S^{\sharp}\circ E = E^{*}\circ S^{\sharp}.
	\end{equation}
	Equations \eqref{quad-intertwine-ad} and \eqref{quad-intertwine-E} show that
	$S^{\sharp}$ intertwines both the adjoint and the $E$-actions.
	Thus $S^{\sharp}$ is a morphism of \ENE-representations from
	$(\g;E,\ad)$ to $(\g^{*};E^{*},\ad^{*})$.
	Since $S$ is nondegenerate, $S^{\sharp}$ is an isomorphism.
\end{proof}

In a quadratic \ENL algebra $(\g,E,S)$, the identification
$S^{\sharp}:\g\to\g^{*}$ allows the coadjoint \ENE-representation to be regarded
as an intrinsic copy of the adjoint one.
This identification will play a central role in the construction of matched
pairs and Manin triples of \ENL algebras in the following section.
	
	\section{Matched pairs and Manin triples of \ENL algebras}\label{sec:mp}
	
	In this section we develop the \ENL analogues of matched pairs and Manin
	triples.  We show that the classical correspondence between bicrossed
	products and quadratic doubles extends naturally to \ENL algebras, yielding two equivalent descriptions of \ENL doubles.
	
	\smallskip
	We begin by recalling the notion of a matched pair of Lie algebras.
	\medskip
	
	A \emph{matched pair of Lie algebras} \cite{Majid,Takeuchi} consists of two Lie algebras
	$(\g,[\cdot,\cdot]_{\g})$ and $(\h,[\cdot,\cdot]_{\h})$, together with
	representations $\rho:\g\to\gl(\h)$ and $\mu:\h\to\gl(\g)$ such that
	\begin{align}
		\label{eq:mp1}
		\rho(x)[\xi,\eta]_{\h}
		&=
		[\rho(x)\xi,\eta]_{\h}
		+ [\xi,\rho(x)\eta]_{\h}
		+ \rho(\mu(\eta)x)\xi
		- \rho(\mu(\xi)x)\eta,\\[0.4em]
		\label{eq:mp2}
		\mu(\xi)[x,y]_{\g}
		&=
		[\mu(\xi)x,y]_{\g}
		+ [x,\mu(\xi)y]_{\g}
		+ \mu(\rho(y)\xi)x
		- \mu(\rho(x)\xi)y,
	\end{align}
	for all $x,y\in\g$ and $\xi,\eta\in\h$.
	Such a matched pair is denoted by $(\g,\h;\rho,\mu)$, or simply $(\g,\h)$.
	
	\smallskip
	
Matched pairs are also referred to as \emph{twilled Lie algebras}
\cite{KM} or \emph{double Lie algebras} \cite{Lu2}.
They admit an equivalent description in terms of bicrossed products:
it is well known \cite{Lu2,Majid} that a matched pair of Lie algebras
$(\g,\h;\rho,\mu)$ determines a Lie algebra structure on the direct sum
$\g\oplus\h$, called the \emph{bicrossed product} and denoted by
$\g\bowtie\h$. The corresponding Lie bracket is given by
\begin{equation}\label{mathched-pair-Lie}
	[x+\xi,\,y+\eta]_{\bowtie}
	=
	\big([x,y]_{\g}+\mu(\xi)y-\mu(\eta)x\big)
	+
	\big([\xi,\eta]_{\h}+\rho(x)\eta-\rho(y)\xi\big),
\end{equation}
for all $x,y\in\g$ and $\xi,\eta\in\h$.
Conversely, if $(\g\oplus\h,[\cdot,\cdot])$ is a Lie algebra such that
$\g$ and $\h$ are Lie subalgebras, then each mixed bracket $[x,\xi]$,
with $x\in\g$ and $\xi\in\h$, decomposes uniquely as
\[
[x,\xi]=\rho(x)\xi-\mu(\xi)x,
\]
and the resulting linear maps $\rho:\g\to\gl(\h)$ and
$\mu:\h\to\gl(\g)$ form a matched pair $(\g,\h;\rho,\mu)$, with the given
Lie bracket on $\g\oplus\h$ coinciding with the bicrossed product bracket
\eqref{mathched-pair-Lie}.
\smallskip

We now introduce the notion of a matched pair in the category of
\ENL algebras by imposing compatibility conditions
between the mutual actions and the equivariant Nijenhuis operators.

\begin{defi}\label{MPR}
	Let $(\g,[\cdot,\cdot]_{\g},E)$ and $(\h,[\cdot,\cdot]_{\h},\huaE)$ be two \ENL
	algebras. Suppose that there are linear maps $\rho:\g\rightarrow\gl(\h)$ and $\mu:\h\rightarrow\gl(\g)$ such that
 $(\h;\huaE,\rho)$ is an \ENE-representation of $(\g,E)$,  $(\g; E,\mu)$ is an \ENE-representation of $(\mathfrak h,\mathcal{E})$. If moreover
 $(\g,\h;\rho,\mu)$ is a matched pair of Lie algebras,
 then we call $((\g,E),(\h,\huaE))$ (or more precisely $((\g,E),(\h,\huaE);\rho,\mu)$)   a \emph{matched pair of \ENL algebras}.
\end{defi}

We have the following equivalent characterization.
\begin{pro}\label{equiv-con}
	Let $(\g,[\cdot,\cdot]_{\g},E)$ and $(\h,[\cdot,\cdot]_{\h},\huaE)$ be two \ENL
	algebras. Then $((\g,E),(\h,\huaE);\rho,\mu)$ is a matched pair of \ENL algebras if and only if $(\g,\h;\rho,\mu)$ is a matched pair of Lie algebras satisfying
\begin{eqnarray}
		\label{Equ-1}
		\huaE(\rho(x)\xi)
		=
		\rho(E x)\xi
		=
		\rho(x)\big(\huaE \xi\big), \\[0.4em]
		\label{Equ-2}
		E(\mu(\xi)x)
		=
		\mu(\huaE\xi)x
	    =
		\mu(\xi)\big(Ex\big),
\end{eqnarray}
	for all $x\in\g$ and $\xi\in\h$.
\end{pro}
For a matched pair of Lie algebras, the direct sum $\g\oplus\h$ equipped
with the bicrossed-product bracket carries a natural Lie algebra
structure.
We now show that, if the matched pair is compatible with equivariant
Nijenhuis operators, this direct sum inherits a canonical \ENL algebra
structure.

\begin{cor}\label{Equ-on}
	Let $((\g,E),(\h,\huaE);\rho,\mu)$ be a matched pair of \ENL algebras.
	Then the linear operator
	\[
	E\oplus \huaE:\g\oplus\h \longrightarrow \g\oplus\h,
	\qquad
	x+\xi \longmapsto Ex+\huaE\xi,
	\]
	is an equivariant Nijenhuis operator on the Lie algebra
	$\g\bowtie\h$.
	Consequently, the triple
	$(\g\bowtie\h,[\cdot,\cdot]_{\bowtie},E\oplus\huaE)$
	is an \ENL algebra.
\end{cor}

\begin{proof}
	For all $x,y\in\g$ and $\xi,\eta\in\h$.
	A direct computation using the bicrossed-product bracket
	\eqref{mathched-pair-Lie} yields
	\begin{align*}
		&(E\oplus \huaE)[x+\xi,y+\eta]_{\bowtie}
		-[(E\oplus \huaE)(x+\xi),y+\eta]_{\bowtie}\\
		&=E[x,y]_{\g}+E(\mu(\xi)y)-E(\mu(\eta)x)
		+\huaE[\xi,\eta]_{\h}+\huaE(\rho(x)\eta)-\huaE(\rho(y)\xi)\\
		&\quad-[Ex,y]_{\g}-\mu(\huaE\xi)y+\mu(\eta)(Ex)
		-[\huaE\xi,\eta]_{\h}-\rho(Ex)\eta+\rho(y)(\huaE\xi).
	\end{align*}
	Since $(\g,E)$ and $(\h,\huaE)$ are \ENL algebras, and the mutual actions
	$\rho$ and $\mu$ satisfy the equivariance conditions
	\eqref{Equ-1}--\eqref{Equ-2}, each group of terms cancels pairwise.
	Hence the above expression vanishes identically, showing that
	$E\oplus\huaE$ commutes with the adjoint action on $\g\bowtie\h$.
	Therefore $E\oplus\huaE$ is an equivariant Nijenhuis operator on
	$\g\bowtie\h$, completing the proof.
\end{proof}
\begin{rmk}\label{equiv}
	Let $(\g,[\cdot,\cdot]_{\g},E)$ and $(\h,[\cdot,\cdot]_{\h},\huaE)$ be \ENL
	algebras.
	Then the triple
	$(\g\bowtie\h,[\cdot,\cdot]_{\bowtie},E\oplus\huaE)$
	is an \ENL algebra if and only if
	$((\g,E),(\h,\huaE);\rho,\mu)$ is a matched pair of \ENL algebras.
\end{rmk}

We now examine how the \ENL matched pair structure behaves under the
associated Nijenhuis deformations. Let $((\g,E),(\h,\huaE);\rho,\mu)$ be a matched pair of \ENL algebras.
The equivariant Nijenhuis operators
$E:\g\to\g$, $\huaE:\h\to\h$, and
$E\oplus\huaE:\g\oplus\h\to\g\oplus\h$
induce the corresponding deformed Lie algebras
$(\g,[\cdot,\cdot]_{E})$, $(\h,[\cdot,\cdot]_{\huaE})$, and
$(\g\bowtie\h,[\cdot,\cdot]_{E\oplus\huaE})$.
For brevity, we denote these by
\[
\g_{E},\qquad \h_{\huaE},\qquad (\g\bowtie\h)_{E\oplus\huaE}.
\]
\begin{thm}\label{demp}
	If $((\g,E),(\h,\huaE);\rho,\mu)$ is a matched pair of \ENL algebras, then
	$(\g_{E},\h_{\huaE};\rho_{(E,\huaE)},\mu_{(E,\huaE)})$ is a matched pair of the
	deformed Lie algebras, where
	\begin{eqnarray}
		\label{eq:rep1}
		\rho_{(E,\huaE)}(x)\xi &=& \rho(Ex)\,\xi,\\
		\label{eq:rep4}
		\mu_{(E,\huaE)}(\xi)x &=& \mu(\huaE\xi)\,x,
	\end{eqnarray}
	for all $x\in\g$ and $\xi\in\h$.
	Moreover, there is a canonical isomorphism of Lie algebras
	\[
	\g_{E}\bowtie \h_{\huaE}\;\cong\;(\g\bowtie\h)_{E\oplus \huaE}.
	\]
\end{thm}

\begin{proof}
	By Corollary~\ref{Equ-on},
	$(\g\bowtie\h,[\cdot,\cdot]_{\bowtie},E\oplus\huaE)$ is an \ENL algebra.
	Hence the deformed bracket on $\g\bowtie\h$ is given by
	\[
	[u,v]_{E\oplus\huaE}=[(E\oplus\huaE)u,\,v]_{\bowtie}.
	\]
	Let $u=x+\xi$ and $v=y+\eta$.
	Using the bicrossed-product bracket \eqref{mathched-pair-Lie},
	the equivariance of $E$ and $\huaE$, and the compatibility identities
	\eqref{Equ-1}--\eqref{Equ-2}, we compute
	\[
	[x+\xi,y+\eta]_{E\oplus\huaE}
	=
	\bigl([x,y]_{E}
	+\mu_{(E,\huaE)}(\xi)y-\mu_{(E,\huaE)}(\eta)x\bigr)
	+
	\bigl([\xi,\eta]_{\huaE}
	+\rho_{(E,\huaE)}(x)\eta-\rho_{(E,\huaE)}(y)\xi\bigr).
	\]
	This is precisely the bicrossed-product bracket associated with the data
	$(\g_{E},\h_{\huaE};\rho_{(E,\huaE)},\mu_{(E,\huaE)})$.
	
	Moreover, for $x\in\g$ and $\xi\in\h$, the mixed bracket satisfies
	\[
	[x,\xi]_{E\oplus\huaE}
	=[Ex,\xi]_{\bowtie}
	=\rho(Ex)\xi-\mu(\xi)(Ex)
	=\rho_{(E,\huaE)}(x)\xi-\mu_{(E,\huaE)}(\xi)x.
	\]
	Therefore,
	$(\g_{E},\h_{\huaE};\rho_{(E,\huaE)},\mu_{(E,\huaE)})$ is a matched pair of Lie
	algebras, and the equality of brackets yields the Lie algebra isomorphism
	$
	\g_{E}\bowtie\h_{\huaE}\;\cong\;(\g\bowtie\h)_{E\oplus\huaE}.
$
\end{proof}
\begin{cor}
	If $((\g,E),(\h,\huaE);\rho,\mu)$ is a matched pair of \ENL algebras, then
	\[
	((\g_{E},E),(\h_{\huaE},\huaE);\rho_{(E,\huaE)},\mu_{(E,\huaE)})
	\]
	is again a matched pair of \ENL algebras.
\end{cor}
\begin{proof}
	By Proposition~\ref{induce-Lie-algebra}, the operator $E\oplus\huaE$ is an
	equivariant Nijenhuis operator on the Lie algebra
	$(\g\bowtie\h)_{E\oplus\huaE}$.
	By Theorem~\ref{demp}, this algebra is canonically isomorphic to
	$\g_{E}\bowtie\h_{\huaE}$.
	Hence $E\oplus\huaE$ defines an equivariant Nijenhuis operator on
	$\g_{E}\bowtie\h_{\huaE}$.
	By Remark~\ref{equiv}, it follows that
	$((\g_{E},E),(\h_{\huaE},\huaE);\rho_{(E,\huaE)},\mu_{(E,\huaE)})$
	is a matched pair of \ENL algebras.
\end{proof}

To complement the matched pair viewpoint, we recall the corresponding
quadratic notion provided by Manin triples of Lie algebras.

\medskip

A \emph{Manin triple} of Lie algebras is a triple
$((\mathfrak d,S),\g,\h)$ consisting of Lie algebras $\g$ and $\h$ together
with a quadratic Lie algebra $(\mathfrak d,S)$ such that:
\begin{itemize}
	\item[\rm (i)] $\g$ and $\h$ are Lie subalgebras of $\mathfrak d$;
	\item[\rm (ii)] the underlying vector space decomposes as
	$\mathfrak d=\g\oplus\h$;
	\item[\rm (iii)] both $\g$ and $\h$ are isotropic with respect to $S$, that is,
	$S(x,y)=0$ for all $x,y\in\g$, and similarly for $\h$.
\end{itemize}
The nondegenerate, invariant, symmetric bilinear form $S$ therefore
identifies $\g$ and $\h$ as complementary maximal isotropic Lie subalgebras
of $\mathfrak d$.

\medskip

It is well known that Manin triples of Lie algebras are in one-to-one
correspondence with matched pairs, providing a quadratic reformulation of
the matched pair construction \cite{D,Lu2,Majid}.
We briefly recall this classical equivalence.

Let $(\g,[\cdot,\cdot]_{\g})$ be a Lie algebra and suppose that its dual
space $\g^{*}$ is endowed with a Lie algebra structure
$[\cdot,\cdot]_{\g^{*}}$.
Then the following two descriptions are equivalent:
\begin{itemize}
	\item[\rm (i)]
	$(\g,\g^{*};\ad^{*},\add^{*})$ is a matched pair of Lie algebras, where
	$\ad^{*}$ is the coadjoint action of $\g$ on $\g^{*}$ and
	$\add^{*}$ is the coadjoint action of $\g^{*}$ on $\g$.
	
	\item[\rm (ii)]
	$((\g\oplus\g^{*},S),\g,\g^{*})$ is a Manin triple, where the bilinear form
	$S\in\otimes^{2}(\g\oplus\g^{*})^{*}$ is given by
	\begin{equation}\label{eq:sp}
		S(x+\xi,\,y+\eta)=\xi(y)+\eta(x),
		\qquad
		\forall x,y\in\g,\; \xi,\eta\in\g^{*}.
	\end{equation}
\end{itemize}

\medskip

We now extend this correspondence to the \ENL setting, using quadratic \ENL
algebras as introduced in Definition~\ref{defi:qua}.

\begin{defi}
	Let $(\g,E)$ and $(\h,\huaE)$ be two \ENL algebras, and let $(\huaG,\frkE,S)$ be a
	quadratic \ENL algebra.
	We say that
	\[
	((\huaG,\frkE,S),(\g,E),(\h,\huaE))
	\]
	is a \emph{Manin triple of \ENL algebras} if the following conditions hold:
	\begin{itemize}
		\item[\rm (i)]
		$(\g,E)$ and $(\h,\huaE)$ are \ENL subalgebras of $(\huaG,\frkE)$; that is,
		$\g$ and $\h$ are Lie subalgebras of $\huaG$ and
		$\frkE|_{\g}=E$, $\frkE|_{\h}=\huaE$.
		
		\item[\rm (ii)]
		The underlying vector space decomposes as
		$\huaG=\g\oplus\h$.
		
		\item[\rm (iii)]
		$\g$ and $\h$ are isotropic subalgebras with respect to the invariant
		bilinear form $S$, that is,
		$S(\g,\g)=0$ and $S(\h,\h)=0$.
	\end{itemize}
\end{defi}

\medskip

In the \ENL setting, the Manin triple structure requires only that the
equivariant Nijenhuis operator on the quadratic double restricts to the
given operators on the two isotropic subalgebras.
The precise relation with matched pairs of \ENL algebras is made explicit
below.

\begin{pro}\label{pro:mt}
	Let $(\huaG,\frkE,S)$ be a quadratic \ENL algebra, and let $(\g,E)$ and
	$(\h,\huaE)$ be \ENL algebras.
	The triple
	\[
	((\huaG,\frkE,S),(\g,E),(\h,\huaE))
	\]
	is a Manin triple of \ENL algebras if and only if
	$((\huaG,S),\g,\h)$ is a Manin triple of Lie algebras and
	\[
	\frkE|_{\g}=E,
	\qquad
	\frkE|_{\h}=\huaE.
	\]
\end{pro}

\begin{thm}\label{matched-manin-equivalent}
	Let $(\g,[\cdot,\cdot]_{\g},E)$ and $(\g^{*},[\cdot,\cdot]_{\g^{*}},E^{*})$ be two \ENL
	algebras, where $\g^{*}$ is the dual space of $\g$.
	The following conditions are equivalent:
	\begin{itemize}
		\item[\rm (i)]
		$((\g,E),(\g^{*},E^{*});\ad^{*},\add^{*})$ is a matched pair of \ENL algebras,
		where $\ad^{*}$ and $\add^{*}$ are the coadjoint representations of $\g$ and
		$\g^{*}$, respectively.
		
		\item[\rm (ii)]
		$((\g\oplus\g^{*},\widehat{E},S),(\g,E),(\g^{*},E^{*}))$ is a Manin triple of \ENL
		algebras, where $\widehat{E}:=E\oplus E^{*}$ is given by
		\[
		\widehat{E}(x+\xi)=Ex+E^{*}\xi,
		\]
		and $S$ is the canonical invariant symmetric bilinear form given by \eqref{eq:sp}.
		\end{itemize}
\end{thm}

\begin{proof}
	(i)$\Rightarrow$(ii).
	Assume that $((\g,E),(\g^{*},E^{*});\ad^{*},\add^{*})$ is a matched pair of \ENL
	algebras.  Then, by Corollary~\ref{Equ-on}, the direct sum
	$(\g\oplus\g^{*},[\cdot,\cdot]_{\bowtie},\widehat{E})$ is an \ENL algebra, where
	$[\cdot,\cdot]_{\bowtie}$ is the bicrossed-product bracket \eqref{mathched-pair-Lie}.
	Moreover, by the classical correspondence between matched pairs and Manin
	triples, $((\g\oplus\g^{*},S),\g,\g^{*})$ is a Manin triple of Lie algebras with
	$S$ given by \eqref{eq:sp}.
	
	It remains to check that $(\g\oplus\g^{*},\widehat{E},S)$ is a quadratic \ENL
	algebra, i.e. that $\widehat{E}$ is $S$-symmetric:
	\[
	S(\widehat{E}(x+\xi),\,y+\eta)-S(x+\xi,\,\widehat{E}(y+\eta))=0.
	\]
	Using \eqref{eq:sp} and the fact that $E^{*}$ is the dual map of $E$, we
	compute
	\begin{eqnarray*}
		&& S(Ex+E^{*}\xi,\,y+\eta)-S(x+\xi,\,Ey+E^{*}\eta) \\
		&=&\langle E^{*}\xi,y\rangle + \langle \eta,Ex\rangle
		- \langle \xi,Ey\rangle - \langle E^{*}\eta,x\rangle
		=0,
	\end{eqnarray*}
	for all $x,y\in\g$ and $\xi,\eta\in\g^{*}$.
	Hence $(\g\oplus\g^{*},\widehat{E},S)$ is a quadratic \ENL algebra and
	$((\g\oplus\g^{*},\widehat{E},S),(\g,E),(\g^{*},E^{*}))$ is a Manin triple of \ENL
	algebras.
	
	\smallskip
	
	(ii)$\Rightarrow$(i).
	Conversely, assume that
	$((\g\oplus\g^{*},\widehat{E},S),(\g,E),(\g^{*},E^{*}))$ is a Manin triple of \ENL
	algebras with $S$ given by \eqref{eq:sp}.  Forgetting the operator, we
	obtain a Manin triple of Lie algebras $((\g\oplus\g^{*},S),\g,\g^{*})$, hence,
	by the classical correspondence, $(\g,\g^{*};\ad^{*},\add^{*})$ is a matched
	pair of Lie algebras and the induced Lie bracket on $\g\oplus\g^{*}$ is the
	bicrossed-product bracket \eqref{mathched-pair-Lie}.
	
	Since $\widehat{E}=E\oplus E^{*}$ is an equivariant Nijenhuis operator on this
	double Lie algebra and restricts to $E$ on $\g$ and to $E^{*}$ on $\g^{*}$,
	Remark~\ref{equiv} yields that
	$((\g,E),(\g^{*},E^{*});\ad^{*},\add^{*})$ is a matched pair of \ENL algebras.
\end{proof}

\section{Structure Theory of ENL Bialgebras}\label{sec:bia}

In this section we study the structure theory of
\emph{\ENL bialgebras}.
These arise as a rigid subclass of the \NL bialgebras introduced in
\cite{Zohreh}, obtained by imposing equivariance of the Nijenhuis operator
with respect to both the Lie algebra and the Lie coalgebra structures. We briefly recall the basic algebraic notions, for more details see \cite{K}.
\smallskip

Lie bialgebras encode a compatibility between a Lie algebra structure and
its dual, governed by a Lie coalgebra.
More precisely, a Lie coalgebra $(\g,\Delta)$ consists of a vector space
$\g$ equipped with a coantisymmetric coproduct
$\Delta:\g\to\g\otimes\g$, that is,
$\Delta=-\tau\circ\Delta$, satisfying the co--Jacobi identity
\[
(\Id+\epsilon+\epsilon^{2})(\Id\otimes\Delta)\Delta=0,
\]
where $\tau:\g\otimes\g\rightarrow\g\otimes\g$ is a flip map and $\epsilon(x\otimes y\otimes z)=z\otimes x\otimes y$.

A Lie bialgebra is a pair $(\g,\g^{*})$ of Lie algebras in duality,
whose compatibility is controlled by such a coproduct through the
condition that $\Delta$ be a $1$-cocycle for the adjoint representation,
namely,
\[
\Delta([x,y]_{\g})=
(\ad_x\otimes\Id+\Id\otimes\ad_x)\Delta(y)
-
(\ad_y\otimes\Id+\Id\otimes\ad_y)\Delta(x),
\qquad \forall x,y\in\g.
\]

A particularly important class is provided by coboundary Lie bialgebras,
for which the cobracket is generated by an element
$r\in\g\otimes\g$ via
\[
\Delta(x)=(\ad_x\otimes\Id+\Id\otimes\ad_x)(r).
\]
In this case, the induced Lie bialgebra structure will be denoted by
$(\g,\Delta_r)$.
\medskip

It is worth noting that if $(\g,\g^{*})$ is a Lie bialgebra, then
$(\g^{*},\g)$ is again a Lie bialgebra.
Another fundamental fact is that every Lie bialgebra admits a canonical
\emph{Drinfel'd double}.
Given a Lie bialgebra $(\g,\g^{*})$, its Drinfel'd double is defined as the
Lie algebra
\[
\mathfrak{d} := \g \oplus \g^{*}
\]
equipped with the unique Lie bracket for which $\g$ and $\g^{*}$ embed as
complementary isotropic Lie subalgebras.
This bracket, usually denoted by $[\cdot,\cdot]_{\bowtie}$, is characterized on
mixed terms by
\[
[x,\xi]_{\bowtie} = \ad^{*}_{x}\xi - \add^{*}_{\xi}x,
\qquad \forall x\in\g,\ \xi\in\g^{*},
\]
where the mutual actions are determined by the Lie bialgebra structure via
\[
\langle \add^{*}_{\xi}x , \eta \rangle
= - \langle x , [\xi,\eta]_{\g^{*}} \rangle,
\qquad
\langle \ad^{*}_{x}\xi , y \rangle
= - \langle \xi , [x,y]_{\g} \rangle,
\]
for all $y\in\g$ and $\eta\in\g^{*}$.
The natural symmetric pairing
\[
S(x+\xi,\; y+\eta)
= \langle \xi , y \rangle + \langle \eta , x \rangle
\]
is invariant under the Drinfel'd double bracket.
Consequently, $(\mathfrak{d},S)$, together with its isotropic Lie
subalgebras $\g$ and $\g^{*}$, forms a \emph{Manin triple}.
In finite dimensions, Drinfel'd showed that this construction yields a
one-to-one correspondence between Lie bialgebras and Manin triples; see
\cite{K,Lu2}.
\medskip

Before turning to the \ENL bialgebra setting, it is useful to recall a more
general framework, namely that of \NL bialgebras introduced in \cite{Zohreh}.
This broader class provides a natural algebraic analogue of
Poisson-Nijenhuis structures on manifolds and allows Nijenhuis operators
to interact with Lie bialgebra structures through deformation of the
underlying brackets.
Recalling this general setting clarifies the structural role played by
equivariance and highlights the additional rigidity inherent in the \ENL
bialgebra framework developed further.

\begin{defi}[\cite{Zohreh}]\label{def:NLbialgebra}
	Let $(\g,N)$ and $(\g^{*},N^{*})$ be \NL algebras, and assume that
	$(\g,\g^{*})$ is a Lie bialgebra with the corresponding cobracket $\Delta$.
	The triple $(\g,\g^{*},N)$ is called an \emph{\NL bialgebra} if the following
	conditions are satisfied:
	\begin{itemize}
		\item[(i)]
		The Lie algebra $\bigl(\g,[\cdot,\cdot]_N\bigr)$ together with
		$\bigl(\g^{*},[\cdot,\cdot]_{\g^{*}}\bigr)$ forms a Lie bialgebra;
		
		\item[(ii)]
		The Lie algebra $\bigl(\g,[\cdot,\cdot]_N\bigr)$ together with
		$\bigl(\g^{*},[\cdot,\cdot]_{N^{*}}\bigr)$ forms a Lie bialgebra.
	\end{itemize}
	Here $[\cdot,\cdot]_N$ and $[\cdot,\cdot]_{N^{*}}$ denote the deformed Lie
	brackets obtained from the original brackets on $\g$ and $\g^{*}$ by
	application of the Nijenhuis operators $N$ and $N^{*}$, respectively.
\end{defi}
\noindent
The first condition is equivalent to requiring that
$\bigl(\g,[\cdot,\cdot]_{\g}\bigr)$ and $\bigl(\g^{*},[\cdot,\cdot]_{N^{*}}\bigr)$
form a Lie bialgebra.
The second condition can be expressed in terms of a concomitant measuring
the interaction between the original cobracket $\Delta$ of the Lie
bialgebra $(\g,\g^{*})$ and the Nijenhuis operator $N$.

\smallskip

\noindent
More precisely, the concomitant $C(\Delta,N)$ is the $(2,1)$-tensor defined by
\[
C(\Delta,N)(\xi,\eta)([x,y]_{\g}) =
\Bigl[
\iota_{\,N^{*}\circ \ad^{*}_{x}}\Delta(y)
- \iota_{\,N^{*}\circ \ad^{*}_{y}}\Delta(x)
- \iota_{\,\ad^{*}_{x}}\Delta(Ny)
+ \iota_{\,\ad^{*}_{y}}\Delta(Nx)
\Bigr](\xi,\eta),
\]
for all $x,y\in\g$ and $\xi,\eta\in\g^{*}$.
Here $\iota_{\phi}$ denotes the natural extension of a linear map
$\Phi:\g\to\g$, with $\phi:=\Phi^{*}:\g^{*}\to\g^{*}$ its transpose, to an
operator
\[
\iota_{\phi}:\wedge^{k}\g\longrightarrow \wedge^{k}\g
\]
defined by
\begin{equation}\label{eq:iota}
	(\iota_{\phi}P)(\xi_{1},\dots,\xi_{k})
	=
	\sum_{i=1}^{k}
	P(\xi_{1},\dots,\phi(\xi_{i}),\dots,\xi_{k}),
	\qquad
	\forall P\in\wedge^{k}\g,\ \xi_{1},\dots,\xi_{k}\in\g^{*}.
\end{equation}
For instance, in this notation the cobracket $\Delta$ is a $1$-cocycle if
\[
\Delta([x_{1},x_{2}])
=
\iota_{\ad^{*}_{x_{1}}}\Delta(x_{2})
-
\iota_{\ad^{*}_{x_{2}}}\Delta(x_{1}).
\]
The vanishing of $C(\Delta,N)$ characterizes condition~{\rm(ii)} in
Definition~\ref{def:NLbialgebra}.
In order to obtain a fully algebraic construction mirroring the theory of
Poisson-Nijenhuis manifolds, it was shown in \cite{Zohreh} that, under
suitable conditions, \NL bialgebras generate hierarchies of \NL bialgebras.

\begin{lem}[\cite{Zohreh}]\label{Zohreh2}
	Let $(\g,\Delta,N)$ be an \NL bialgebra.
	Then $(\g,\Delta,N)$ generates a hierarchy of \NL bialgebras provided that
	one of the following conditions holds:
	\begin{itemize}
		\item $N:\g\to\g$ is an equivariant Nijenhuis operator;
		\item $N^{*}:\g^{*}\to\g^{*}$ is an equivariant Nijenhuis operator;
		\item the $1$-cocycle $\Delta$ associated with the Lie bialgebra
		$(\g,\g^{*})$ is a coboundary.
	\end{itemize}
\end{lem}
\noindent
Let us stress, by means of a concrete example, that in the \NL bialgebra
setting the existence of a hierarchy does not require equivariance of
both Nijenhuis operators.
Equivariance of either $N$ or $N^{*}$ already suffices.
The following example gives an explicit illustration of
this phenomenon.

\begin{ex}\label{ex-1equiv}
	Let $\g$ be the direct sum of two copies of the $2$-dimensional solvable
	Lie algebra.
	With respect to a basis $\{X_{1},X_{2},X_{3},X_{4}\}$, the nonzero Lie
	brackets are
	\[
	[X_{1},X_{2}] = X_{2},
	\qquad
	[X_{3},X_{4}] = X_{4}.
	\]
	Define a cobracket $\Delta:\g\to\wedge^{2}\g$ by
	\[
	\Delta(X_{1})=0,\qquad
	\Delta(X_{2})=0,\qquad
	\Delta(X_{3})=X_{3}\wedge X_{4},\qquad
	\Delta(X_{4})=X_{3}\wedge X_{4},
	\]
	and a linear operator $N:\g\to\g$ by
	\[
	N(X_{1})=X_{1},\qquad
	N(X_{2})=X_{2},\qquad
	N(X_{3})=X_{4},\qquad
	N(X_{4})=0.
	\]
	The operator $N$ is a Nijenhuis operator, and one checks that
	$(\g,\Delta,N)$ is an \NL bialgebra.
	Moreover, $N$ is equivariant, so Lemma~\ref{Zohreh2} implies that this \NL
	bialgebra generates a hierarchy of Lie bialgebras.
	However, the dual operator $N^{*}$ is not equivariant, since
	\[
	N^{*}[X^{*}_{3},X^{*}_{4}] \neq [X^{*}_{3},N^{*}(X^{*}_{4})],
	\]
 where $\{X^*_{1},X^*_{2},X^*_{3},X^*_{4}\} $ be the dual basis of the dual vector space $\g^*.$
	The only nonzero deformed brackets induced by $N$ and $N^{*}$ are
	\[
	[X_{1},X_{2}]_{N} = X_{2},
	\qquad
	[X^{*}_{3},X^{*}_{4}]_{N^{*}} = -X^{*}_{3}.
	\]
\end{ex}
\medskip

A class of \NL bialgebras $(\g,\g^{*})$ relevant for the present work consists
of those equipped with equivariant Nijenhuis operators $E$ on $\g$ and
$E^{*}$ on $\g^{*}$.
In what follows, we focus exclusively on this class.

\begin{defi}
	An \emph{\ENL bialgebra} $(\g,\g^{*},E)$ is a Lie bialgebra $(\g,\g^{*})$ such
	that $(\g,[\cdot,\cdot],E)$ and $(\g^{*},[\cdot,\cdot]_{\g^{*}},E^{*})$ are
\ENL algebras.
\end{defi}

\begin{cor}
	If $(\g,\g^{*},E)$ is an \ENL bialgebra, then $(\g^{*},\g,E^{*})$ is also an
	\ENL bialgebra.
\end{cor}

\begin{lem}\label{ad-NL}
	Every \ENL bialgebra $(\g,\g^{*},E)$ is an \NL bialgebra in the sense of
	Definition~\ref{def:NLbialgebra}.
	Moreover, \ENL bialgebras generate a hierarchy of Lie bialgebras.
\end{lem}

\begin{proof}
By Remark~\ref{equiv-ad}, the concomitant $C(\Delta,E)$ of the cobracket
$\Delta$ with respect to $E$ vanishes identically.
It therefore remains to show that
$
\bigl((\g,[\cdot,\cdot]_{E}),(\g^{*},[\cdot,\cdot]_{\g^{*}})\bigr)
$
forms a Lie bialgebra.
Since $\ad^{E}_{x}=\ad_{Ex}$, the $1$-cocycle condition for $\Delta$ with
respect to the deformed bracket $[\cdot,\cdot]_{E}$ reads
\[
\Delta([x,y]_{E})
=
\iota_{\ad^{*}_{Ex}}\Delta(y)
-
\iota_{\ad^{*}_{Ey}}\Delta(x).
\]
On the other hand, $\Delta$ is a $1$-cocycle for the original bracket
$[\cdot,\cdot]_{\g}$, hence
\[
\Delta([x,Ey]_{\g})
=
\iota_{\ad^{*}_{x}}\Delta(Ey)
-
\iota_{\ad^{*}_{Ey}}\Delta(x).
\]
Using the equivariance of $E$ and $E^{*}$ (see \eqref{adn-equiv}), we obtain
\[
\iota_{\ad^{*}_{Ex}}\Delta(y)
=
\iota_{E^{*}\circ\ad^{*}_{x}}\Delta(y)
=
\iota_{\ad^{*}_{x}}\Delta(Ey),
\]
	and similarly for the second term.
	Thus the two cocycle identities coincide, and $\Delta$ is also a
	$1$-cocycle for the deformed bracket $[\cdot,\cdot]_{E}$.
	Therefore,
	$
	\bigl((\g,[\cdot,\cdot]_{E}),(\g^{*},[\cdot,\cdot]_{\g^{*}})\bigr)
	$ is a Lie bialgebra, which proves that $(\g,\g^{*},E)$ is an \NL bialgebra.
	The hierarchy property follows immediately from
	Lemma~\ref{Zohreh2}.
\end{proof}
In contrast with Example~\ref{ex-1equiv}, the following example exhibits a
genuine \ENL bialgebra.
\begin{ex}
	Consider the Lie algebra in Example~\ref{ex-1equiv} and equip it with the
	Lie bialgebra structure determined by the $1$-cocycle
	\[
	\Delta(X_{1}) = 0, \qquad
	\Delta(X_{2}) = 2 X_{1}\wedge X_{2}, \qquad
	\Delta(X_{3}) = X_{3}\wedge X_{4}, \qquad
	\Delta(X_{4}) = 0.
	\]
	Define a linear operator $E:\g\to\g$ by
	\[
	E(X_{1})=X_{1}, \qquad E(X_{2})=X_{2}, \qquad
	E(X_{3})=0,   \qquad E(X_{4})=0.
	\]
	Then $E$ and its transpose $E^{*}$ are equivariant Nijenhuis operators on
	$\g$ and $\g^{*}$, respectively.
	Hence $(\g,\Delta,E)$ defines an \ENL bialgebra.
\end{ex}
In what follows, we show that the classical correspondence between Lie
bialgebras and matched pairs extends directly to the \ENL setting.

\begin{thm}\label{bialgebra-matched}
	Let $(\g,[\cdot,\cdot]_{\g},E)$ and $(\g^{*},[\cdot,\cdot]_{\g^{*}},E^{*})$
	be \ENL algebras.
	Then $(\g,\g^{*},E)$ is an \ENL bialgebra if and only if
	$((\g,E),(\g^{*},E^{*});\ad^{*},\add^{*})$ is a matched pair of \ENL
	algebras.
\end{thm}

\begin{proof}
	Assume first that $(\g,\g^{*},E)$ is an \ENL bialgebra.
	By Corollary~\ref{thm:coadjoint}, $(\g^{*};E^{*},\ad^{*})$ is an
	\ENE-representation of the \ENL algebra $(\g,E)$, while
	$(\g;E,\add^{*})$ is an \ENE-representation of the \ENL algebra
	$(\g^{*},E^{*})$.
	Moreover, it is well known that $(\g,\g^{*})$ is a Lie bialgebra if and
	only if $(\g,\g^{*};\ad^{*},\add^{*})$ is a matched pair of Lie algebras.
	It follows that the \ENE-equivariance conditions in
	Definition~\ref{MPR} are automatically satisfied, and hence
	$((\g,E),(\g^{*},E^{*});\ad^{*},\add^{*})$ is a matched pair of \ENL
	algebras.
	
	Conversely, suppose that
	$((\g,E),(\g^{*},E^{*});\ad^{*},\add^{*})$ is a matched pair of \ENL
	algebras.
	Then $(\g,\g^{*})$ is a matched pair of Lie algebras, and therefore a
	Lie bialgebra.
	The \ENE-equivariance of $E$ and $E^{*}$ with respect to the mutual
	actions implies that $(\g,\g^{*},E)$ is an \ENL bialgebra.
\end{proof}
\begin{cor}\label{cor:dd}
	Let $(\g,\g^{*},E)$ be an \ENL bialgebra.
	Then the Drinfel'd double
$
	\mathfrak d := \g \bowtie \g^{*}
$
	together with the linear map
	\[
	\frkE : \g \oplus \g^{*} \longrightarrow \g \oplus \g^{*},
	\qquad
	\frkE(x+\xi) = Ex + E^{*}\xi, \quad \forall x\in \g,\xi \in \g^*
	\]
	forms an \ENL algebra $(\mathfrak d,\frkE)$.
\end{cor}

\begin{proof}
	By Theorem~\ref{bialgebra-matched}, the pair
	$((\g,E),(\g^{*},E^{*});\ad^{*},\add^{*})$
	is a matched pair of \ENL algebras.
	Hence the bicrossed product $\mathfrak d=\g\bowtie\g^{*}$ carries a natural
	Lie algebra structure.
	By Corollary~\ref{Equ-on}, the operator
	$\frkE=E\oplus E^{*}$ is an equivariant Nijenhuis operator on
	$\mathfrak d$.
	Therefore $(\mathfrak d,\frkE)$ is an \ENL algebra.
\end{proof}

\smallskip

The next step is to describe the dual Lie algebra of the Drinfel'd double
and to show that the \ENL structure extends to the full double Lie
bialgebra.

\smallskip

Given a Lie bialgebra $(\g,\g^{*})$, its Drinfel'd double is the Lie algebra
$
\mathfrak d := \g \bowtie \g^{*},
$
characterized by the fact that $\g$ and $\g^{*}$ embed as complementary
isotropic Lie subalgebras with respect to the canonical symmetric
bilinear form on $\g\oplus\g^{*}$.
A canonical tensor in $\mathfrak d\otimes\mathfrak d$ plays a central role
in describing the dual Lie bialgebra structure.

Let $\{e_{1},\dots,e_{n}\}$ be a basis of $\g$ with dual basis
$\{\xi_{1},\dots,\xi_{n}\}$ of $\g^{*}$.
Define the canonical $r$-matrix
\[
r = \sum_{i=1}^{n} e_{i}\otimes \xi_{i}
\in \g\otimes\g^{*} \subset \mathfrak d\otimes\mathfrak d .
\]
This element induces a Lie bracket on $\mathfrak d^{*}$, denoted by
$[\cdot,\cdot]_{r}$, given by
\[
[\xi + x,\;\eta + y]_{r}
=
\bigl(-[\xi,\eta]_{\g^{*}},\; [x,y]_{\g}\bigr),
\qquad \forall x,y\in\g,\;\xi,\eta\in\g^{*}.
\]
The resulting Lie algebra is denoted by $\mathfrak d^{*}_{r}$, and the
pair $(\mathfrak d,\mathfrak d^{*}_{r})$ forms a \emph{quasi-triangular Lie
bialgebra}; see \cite{RS}.
We now show that the \ENL structure on $(\g,\g^{*})$ extends naturally to
this double Lie bialgebra.
\begin{pro}\label{Double-r-bia}
	Let $(\g,\g^{*},E)$ be an \ENL bialgebra.
	Then the triple $(\mathfrak d,\mathfrak d^{*}_{r},\frkE)$ is also an \ENL
	bialgebra, where
	\[
	\frkE : \mathfrak d=\g\oplus\g^{*} \longrightarrow \g\oplus\g^{*},
	\qquad
	\frkE(x+\xi)=Ex+E^{*}\xi .
	\]
\end{pro}

\begin{proof}
	By Corollary~\ref{cor:dd}, the operator
	$\frkE=E\oplus E^{*}$ is an equivariant Nijenhuis operator on the
	Drinfel'd double $\mathfrak d=\g\bowtie\g^{*}$.
	It therefore remains to show that the dual operator
	\[
	\frkE^{*} : \mathfrak d^{*}_{r} \longrightarrow \mathfrak d^{*}_{r},
	\qquad
	\frkE^{*}(\xi + x) = E^{*}\xi + Ex ,
	\]
	is an equivariant Nijenhuis operator with respect to the bracket
	$[\cdot,\cdot]_{r}$.
		Since
	\[
	[\xi + x,\;\eta + y]_{r}
	=
	\bigl(-[\xi,\eta]_{\g^{*}},\; [x,y]_{\g}\bigr),
	\]
	the Lie bracket on $\mathfrak d^{*}_{r}$ decomposes as the direct sum of
	the Lie brackets on $\g^{*}$ and on $\g$.
	The Nijenhuis torsion of $\frkE^{*}$ therefore splits into the torsions of
	$E^{*}$ on $(\g^{*},[\cdot,\cdot]_{\g^{*}})$ and of $E$ on
	$(\g,[\cdot,\cdot]_{\g})$, both of which vanish since
	$(\g,\g^{*},E)$ is an \ENL bialgebra.
	
	The equivariance conditions for $\frkE^{*}$ reduce in the same way to the
	equivariance identities for $E$ and $E^{*}$ on $\g$ and $\g^{*}$, respectively.
	Hence $\frkE^{*}$ is an equivariant Nijenhuis operator on
	$\mathfrak d^{*}_{r}$, and the claim follows.
\end{proof}

\section{Equivariant Nijenhuis structures on Rota-Baxter Lie algebras: \ENLE-RB algebras}
\label{sec:bia2}

Equivariant Nijenhuis operators interact naturally with Lie-type structures
beyond the matched pair and Manin triple settings.
A particularly rigid and geometrically meaningful framework is provided by
quadratic Rota-Baxter Lie algebras, where splitting operators,
invariant bilinear forms, and solutions of the classical Yang-Baxter equation
are tightly intertwined.

In this section, we investigate equivariant Nijenhuis operators that are
compatible with quadratic Rota-Baxter Lie algebra structures.
Relying on the correspondence between quadratic Rota-Baxter Lie algebras
and factorizable Lie bialgebras, we show that such compatible pairs
naturally give rise to \ENL bialgebra structures.
\medskip

We begin by recalling the necessary background on coboundary Lie bialgebras,
classical $r$-matrices, and Rota-Baxter Lie algebras.
\smallskip

Let $(\g,[\cdot,\cdot]_\g)$ be a Lie algebra and let $r\in\g\otimes\g$.
Viewing $r$ as a bilinear form on $\g^*$, it induces a linear map
\[
r_+:\g^*\to\g, \qquad r_+(\xi)=r(\xi,\cdot), \quad  \forall\xi\in\g^*,
\]
together with its adjoint $r_-:=-r_+^*$.
These maps provide a fundamental link between tensorial data on $\g$
and algebraic structures on its dual. The element $r$ defines a natural candidate cobracket on $\g$ via
\[
\Delta=\Delta_r,\qquad
\Delta_r(x)=(\ad_x\otimes\Id+\Id\otimes\ad_x)r,
\qquad \forall x\in\g.
\]
The compatibility of $\Delta_r$ with the Lie bracket is governed by the
Schouten bracket
\[
\lcf r,r \rcf
= [r_{12},r_{13}] + [r_{13},r_{23}] + [r_{12},r_{23}],
\]
where $r_{ij}$ denotes the standard embeddings of $r$ into
$U(\g)^{\otimes3}$.
The condition $\lcf r,r \rcf=0$ is the classical
\emph{Yang-Baxter equation}, and its solutions are called
\emph{classical $r$-matrices}.
\smallskip

If $r$ satisfies the classical Yang-Baxter equation and  the symmetric part of $r$ is  $\ad$-invariant:
\[
(\ad_x\otimes\Id+\Id\otimes\ad_x)(r+\sigma(r))=0,
\qquad \forall x\in\g,
\] where $\sigma$ denotes the flip map on $\g\otimes\g$, then the bracket on
$\g^*$ defined by
\[
[\xi,\eta]_r
= \ad_{r_+\xi}^*\eta - \ad_{r_-\eta}^*\xi,
\qquad \forall\xi,\eta\in\g^*,
\]
satisfies the Jacobi identity and thus endows $\g^*$ with a Lie algebra
structure \cite{CP}, denoted by $\g_r^*$.
Moreover, in this case $\Delta_r$ is a $1$-cocycle in the Chevalley-Eilenberg
cohomology of $\g$ with coefficients in the $\g$-module $\g\otimes\g$,
for the adjoint representation.
Under these conditions, $(\g,\Delta_r)$ is a coboundary Lie bialgebra,
and the associated pair $(\g,\g_r^*)$ is called \emph{quasi-triangular}.
In the skew-symmetric case $r+\sigma(r)=0$, the invariance condition is
automatic and the structure is called \emph{triangular}.

\smallskip

A distinguished role is played by the operator
\[
I:=r_+ - r_- : \g^* \to \g,
\]
which measures the deviation of $r$ from skew-symmetry.
The $\ad$-invariance of $r+\sigma(r)$ is equivalent to the equivariance
property
\[
I \circ \ad_x^* = \ad_x \circ I, \qquad \forall x \in \g.
\]
Equivalently, $\tfrac12 I$ corresponds to the symmetric part
$\tfrac12(r+\sigma(r))$ under the natural identification
$\g \otimes \g \simeq \Hom(\g^*,\g)$.
If $r$ is skew-symmetric, then $I=0$.
When $I$ is nondegenerate, it yields a canonical identification between
$\g$ and $\g^*$, and the corresponding quasi-triangular Lie bialgebra is
called \emph{factorizable}, see \cite{RS}.
\smallskip

Rota-Baxter operators provide a natural source of such structures.
A \emph{Rota-Baxter operator of weight $\lambda$} on $\g$ is a linear map
$B:\g \to \g$ satisfying
\[
[Bx,By]_\g
= B\big([Bx,y]_\g + [x,By]_\g + \lambda [x,y]_\g\big),
\qquad \forall x,y \in \g.
\]
Associated to $B$ is the \emph{descendant Lie bracket}
\[
[x,y]_B = [Bx,y]_\g + [x,By]_\g + \lambda [x,y]_\g,
\]
which turns $\g$ into a Lie algebra $\g_B$, called the
\emph{descendant Lie algebra}, for which $B$ is a Lie algebra homomorphism from
$\g_B$ to $\g$.
\medskip

Equivariant Nijenhuis operators are compatible with the Rota-Baxter
descendant construction.
\begin{pro}\label{rota-baxterLie-ave}
	Let $(\g,[\cdot,\cdot]_\g,B)$ be a Rota-Baxter Lie algebra of weight $\lambda$,
	and let $E:\g\to\g$ be an equivariant Nijenhuis operator on
	$(\g,[\cdot,\cdot]_\g)$.
	If $E\circ B = B\circ E$, then $E$ is an equivariant Nijenhuis operator on
	the descendant Lie algebra $(\g,[\cdot,\cdot]_B)$.
\end{pro}
\begin{proof}
	Let $x,y\in\g$.
	Using $E\circ B = B\circ E$ and the equivariance of $E$ with respect to
	$[\cdot,\cdot]_\g$, we have
	\[
	E[x,y]_B
	= [E(Bx),y]_\g + [Ex,By]_\g + \lambda[Ex,y]_\g
	= [Ex,y]_B,
	\]
	so $E$ is an equivariant Nijenhuis operator on $[\cdot,\cdot]_B$.
\end{proof}
\smallskip

When a Rota-Baxter Lie algebra carries additional quadratic data, the link to
Lie bialgebras becomes particularly rigid.
A \emph{quadratic Rota-Baxter Lie algebra} consists of a Rota-Baxter operator
$B$ of weight $\lambda$ together with a nondegenerate invariant symmetric
bilinear form $S$ satisfying
\[
S(x,By)+S(Bx,y)+\lambda S(x,y)=0,
\qquad \forall x,y\in\g.
\]
In this setting, the bilinear form $S$ induces an identification
$\huaI_S:\g^*\to\g$, characterized by
\[
\langle \huaI_S^{-1}x,y\rangle := S(x,y),
\qquad \forall x,y\in\g.
\]
Following the construction of ~\cite{Lang-Sheng}, the operator
\[
r^{B,S}_+ := \frac{1}{\lambda}(B+\lambda\Id)\circ\huaI_S,
\qquad
r^{B,S}_+(\xi)=r^{B,S}(\xi,\cdot), \quad \forall\,\xi\in\g^*,
\]
defines a solution of the classical Yang-Baxter equation.
The associated coboundary Lie bialgebra $(\g,\g^*_{r^{B,S}})$ is
quasi-triangular and factorizable.
Moreover, this construction yields a one-to-one correspondence between
quadratic Rota-Baxter Lie algebras and factorizable Lie bialgebras.
\medskip

This correspondence provides the natural background for the study of
equivariant Nijenhuis operators in the Rota-Baxter setting, which we now
develop.
\begin{defi}\label{def:QENRB}
	Let $(\g,B,S)$ be a quadratic Rota-Baxter Lie algebra of weight $\lambda$,
	and let $E:\g\to\g$ be a linear operator.
	A \emph{quadratic equivariant Nijenhuis Rota-Baxter Lie (\ENLE-{\rm RB}) algebra}
	consists of the data $(\g,B,S,E)$ such that
	\begin{itemize}
		\item[(i)] $(\g,E,S)$ is a quadratic \ENL algebra;
		\item[(ii)] $E$ commutes with the Rota-Baxter operator $B$, namely
		$E\circ B=B\circ E$.
	\end{itemize}
\end{defi}
\begin{thm}\label{thm:FL}
	Let $(\g,B,S,E)$ be a quadratic \ENLE-{\rm RB} algebra of weight $\lambda\neq 0$.
	Then the triple $(\g,\g_{r^{B,S}}^*,E)$ is an \ENL bialgebra.
\end{thm}
\begin{proof}
	The quadratic Rota-Baxter Lie algebra $(\g,B,S)$ determines a factorizable
	Lie bialgebra $(\g,\g_{r^{B,S}}^*)$ associated to the classical $r$-matrix
	$r^{B,S}$.
	Since $(\g,B,S,E)$ is a quadratic \ENLE-{\rm RB} algebra, the operator $E$ is an
	equivariant Nijenhuis operator on $(\g,[\cdot,\cdot]_\g)$ and is
	$S$-symmetric.
	Moreover, $E$ commutes with the Rota-Baxter operator $B$.
	It therefore remains to show that the dual operator $E^*$ is equivariant on
	the Lie algebra $\g_{r^{B,S}}^*$.
	
	By the construction in ~\cite{Lang-Sheng}, the nondegenerate invariant bilinear form
	$S$ induces an identification $\huaI_S:\g^*\to\g$ with inverse
	$\huaI_S^{-1}:\g\to\g^*$, and the normalized map
	\[
	\Phi:=\frac{1}{\lambda}\huaI_S:\ (\g^*,[\cdot,\cdot]_{r^{B,S}})
	\longrightarrow (\g,[\cdot,\cdot]_B)
	\]
	is a Lie algebra isomorphism.
	Since $(\g,E,S)$ is quadratic, the operator $E$ is $S$-symmetric, meaning
	\[
	\langle \huaI_S^{-1}Ex,y\rangle
	=\langle \huaI_S^{-1}x,Ey\rangle,
	\qquad \forall\,x,y\in\g.
	\]
	Equivalently,
	\[
	\huaI_S^{-1}\circ E = E^*\circ \huaI_S^{-1},
	\qquad\text{or}\qquad
	E\circ \huaI_S = \huaI_S\circ E^*,
	\]
	and hence $\Phi\circ E^* = E\circ \Phi$.
	
	Since $E$ commutes with $B$, Proposition~\ref{rota-baxterLie-ave} implies that
	$E$ is an equivariant Nijenhuis operator on the descendant Lie algebra
	$(\g,[\cdot,\cdot]_B)$.
	We now show that $E^*$ is equivariant with respect to the bracket
	$[\cdot,\cdot]_{r^{B,S}}$.
	For all $\xi,\eta\in\g^*$, we compute
	\[
	\begin{aligned}
		\Phi\big(E^*[\xi,\eta]_{r^{B,S}}\big)
		&=E\Phi([\xi,\eta]_{r^{B,S}})
		=E[\Phi(\xi),\Phi(\eta)]_B \\
		&=[E\Phi(\xi),\Phi(\eta)]_B
		=[\Phi(E^*\xi),\Phi(\eta)]_B
		=\Phi\big([E^*\xi,\eta]_{r^{B,S}}\big).
	\end{aligned}
	\]
	Since $\Phi$ is injective, this yields
	\[
	E^*[\xi,\eta]_{r^{B,S}}=[E^*\xi,\eta]_{r^{B,S}}.
	\]
	Therefore $E^*$ is an equivariant operator on the Lie algebra
	$\g_{r^{B,S}}^*$, and consequently $(\g,\g_{r^{B,S}}^*,E)$ is an \ENL
	bialgebra.
\end{proof}
\begin{rmk}
We can also obtain \textup{ENL} bialgebras from a quadratic 
\textup{ENL}-{\rm RB} algebra $(\mathfrak g,B,S,E)$ of weight $\lambda=0$.
In this case, $r^{B,S}_+\in \mathfrak g\otimes \mathfrak g$ is defined by
$r^{B,S}_+ := B\circ \mathcal{I}_S,$
and the resulting Lie bialgebra $(\mathfrak g,\mathfrak g_{r^{B,S}}^*)$ is triangular.
\end{rmk}
\section{\EN $r$-matrices and the classical Yang-Baxter equation}
\label{sec:enr}

Coboundary Lie bialgebras provide one of the main sources of Lie bialgebra
structures, arising from solutions of the classical Yang-Baxter equation.
In the presence of an equivariant Nijenhuis operator, this construction becomes
more rigid: only those $r$-matrices compatible with the operator $E$ can
generate \ENL bialgebras.

In this section, we adapt the classical Yang-Baxter equation to the \ENL
setting. We introduce the notion of \EN $r$-matrices and show that they
parametrize coboundary ENL bialgebras, playing the same
generating role as classical $r$-matrices in the ordinary theory.

\medskip

Let $(\g,\Delta)$ be a Lie coalgebra. We use Sweedler's notation
\[
\Delta(x)=\sum x_{(1)}\otimes x_{(2)}, \qquad \forall x\in\g.
\]

\begin{defi}
	An \emph{equivariant Nijenhuis Lie (\ENLE) coalgebra} is a Lie coalgebra
	$(\g,\Delta)$ equipped with a linear operator $E:\g\to\g$ such that
	\begin{equation}\label{Lie-coalgebra-Equivariant-Nijenhuis}
		\Delta\circ E=(E\otimes \Id)\circ \Delta .
	\end{equation}
\end{defi}

Equivalently, for all $x\in\g$,
\[
\Delta(E x)=\sum E(x_{(1)})\otimes x_{(2)}.
\]
We denote such a structure by the triple $(\g,\Delta,E)$.

\begin{rmk}
	Condition~\eqref{Lie-coalgebra-Equivariant-Nijenhuis} expresses an
	equivariance requirement rather than the usual coalgebra morphism condition
	$\Delta\circ E=(E\otimes E)\circ\Delta$.
	This choice is dictated by duality with \ENL algebras
	and by compatibility with coboundary constructions arising from
	classical $r$-matrices.
\end{rmk}

\begin{defi}
	An \emph{\ENL bialgebra} $(\g,\Delta,E)$ is said to be \emph{coboundary} if
	there exists an element $r\in \g\otimes\g$ such that
	\begin{equation}\label{r-coboundary}
		\Delta=\Delta_r,
		\qquad
		\Delta_r(x):=(\ad_x\otimes \Id+\Id\otimes \ad_x)(r),
		\quad x\in\g.
	\end{equation}
\end{defi}

This is the classical notion of a coboundary Lie bialgebra.
In the equivariant Nijenhuis setting, the novelty lies in the additional
compatibility conditions imposed by the operator $E$ on the corresponding
$r$-matrix.

\begin{pro}\label{r-matrix-condition}
	Let $(\g,[\cdot,\cdot]_{\g},E)$ be an \ENL algebra and let $r\in\g\otimes\g$.
	Assume that the cobracket $\Delta=\Delta_r$ defined by
	\eqref{r-coboundary} endows $(\g,\g^*)$ with a coboundary Lie bialgebra structure.
	Then $(\g^*,[\cdot,\cdot]_{\g^*},E^*)$ is an \ENL algebra if and only if, for
	all $x\in\g$,
	\begin{equation}\label{Equ-r-matrix-1-strong}
		(\ad_{E x}\otimes \Id+\Id\otimes \ad_{E x})(r)
		=
		(E\otimes \Id)\,(\ad_x\otimes \Id+\Id\otimes \ad_x)(r).
	\end{equation}
\end{pro}

\begin{proof}
	Recall that the Lie bracket on $\g^*$ induced by $\Delta$ is given by
	\begin{equation}\label{dual-bracket}
		\langle [\alpha,\beta]_{\g^*},x\rangle
		=
		\langle \alpha\otimes\beta,\Delta(x)\rangle,
		\qquad \forall \alpha,\beta\in\g^*,\ x\in\g.
	\end{equation}
	Hence $(\g^*,[\cdot,\cdot]_{\g^*},E^*)$ is an \ENL algebra if and only if
	\begin{equation}\label{ENL-on-dual}
		[E^*\alpha,\beta]_{\g^*}=E^*[\alpha,\beta]_{\g^*},
		\qquad \forall\,\alpha,\beta\in\g^*.
	\end{equation}
	
	Pairing \eqref{dual-bracket} with $E x\in\g$ and using
	\eqref{ENL-on-dual}, we obtain
	\[
	\langle \alpha\otimes\beta,\Delta(E x)\rangle
	=
	\langle \alpha\otimes\beta,(E\otimes\Id)\Delta(x)\rangle,
	\qquad \forall\,\alpha,\beta\in\g^*,\ x\in\g,
	\]
	which is equivalent to the equivariance condition
	$$\Delta\circ E=(E\otimes\Id)\circ\Delta\,.$$
	For $\Delta=\Delta_r$ given by \eqref{r-coboundary}, this identity is
	equivalent to \eqref{Equ-r-matrix-1-strong} by a direct computation.
\end{proof}

Given an \ENL algebra $(\g,[\cdot,\cdot]_{\g},E)$ and an element $r\in\g\otimes\g$,
the coboundary formula~\eqref{r-coboundary} defines a canonical candidate
cobracket $\Delta_r$.
Proposition~\ref{r-matrix-condition} yields the following results.

\begin{defi}
	Let $(\g,[\cdot,\cdot]_{\g},E)$ be an \ENL algebra.
	An element $r\in\g\otimes\g$ is said to satisfy the \emph{classical
		Yang-Baxter equation in the equivariant Nijenhuis setting} if:
	\begin{itemize}
		\item[\rm(i)]
		$r$ satisfies the classical Yang-Baxter equation
		\[
		\lcf r,r\rcf =0;
		\]
		\item[\rm(ii)]
		$r$ is compatible with the equivariant Nijenhuis operator $E$, that is,
		\begin{equation}\label{Equ-LIE-Yang-Baxter-2}		
		(\Id\otimes E-E\otimes\Id)(r)=0.
	\end{equation}
	\end{itemize}
	Such an element $r$ is called an \emph{\EN $r$-matrix}.
\end{defi}

\begin{thm}\label{Equ-bialgebra-cybe}
	Let $(\g,[\cdot,\cdot]_{\g},E)$ be an \ENL algebra and let $r\in\g\otimes\g$
	be an \EN $r$--matrix.
	Assume moreover that the symmetric part of $r$ is $\ad$-invariant, namely
	\[
	(\ad_x\otimes\Id+\Id\otimes\ad_x)(r+\sigma (r))=0,
	\qquad \forall x\in\g.
	\]
	Then the cobracket $\Delta_r$ defined by \eqref{r-coboundary} endows $\g$ with
	a coboundary \ENL bialgebra structure.
\end{thm}

\begin{proof}
	By the $\ad$-invariance of $r+\sigma (r)$ together with the classical
	Yang-Baxter equation $\lcf r,r\rcf=0$, the cobracket $\Delta_r$ defines a
	coboundary Lie bialgebra structure on $\g$ in the usual sense.
	
	It remains to verify the \ENL compatibility condition
	$\Delta_r\circ E=(E\otimes\Id)\circ\Delta_r$.
	Using the \ENL algebra identity $[Ex,y]_{\g}=E([x,y]_{\g})$, we compute for
	$x\in\g$,
	\[
	\Delta_r(Ex)
	=(\ad_{Ex}\otimes\Id+\Id\otimes\ad_{Ex})(r)
	=(E\otimes\Id)(\ad_x\otimes\Id+\Id\otimes\ad_x)(r)
	=(E\otimes\Id)\Delta_r(x),
	\]
	where the second equality follows from the \EN $r$-matrix condition
	$(\Id\otimes E-E\otimes\Id)(r)=0$.
	Hence $(\g,\Delta_r,E)$ is an \ENL bialgebra, which is coboundary by
	construction.
\end{proof}

\begin{cor}
	Let $(\g,[\cdot,\cdot]_{\g},E)$ be an \ENL algebra and let $r\in\g\otimes\g$
	be a \emph{triangular} \EN $r$--matrix, i.e.\ $r+\sigma (r)=0$.
	Then $r$ canonically generates a coboundary \ENL bialgebra structure on
	$\g$.
\end{cor}

\section{\ENE-Relative Rota-Baxter operators on \ENL algebras}
\label{sec.6}

Relative Rota-Baxter operators play a central role in Lie algebra theory.
They underlie the construction of descendant Lie algebras, matched pairs,
and provide a systematic source of solutions of the classical
Yang-Baxter equation on semidirect products. See also \cite{G3,Goncharov} for more applications.

In Section~\ref{sec:enr}, we showed that \EN $r$-matrices generate
coboundary \ENL bialgebras.
In the present section, we explain how such $r$-matrices arise naturally
from relative Rota-Baxter operators compatible with an equivariant
Nijenhuis operator.
After recalling the classical notion of a relative Rota-Baxter operator,
we introduce its \ENL analogue and show how the equivariance condition
imposes additional rigidity on the induced algebraic structures.

\medskip

Let $(\g,[\cdot,\cdot]_{\g})$ be a Lie algebra and let $(W;\rho)$ be a
representation of $\g$.
A linear map $K:W\to\g$ is called a \emph{relative Rota--Baxter operator}
(or an $\mathcal O$-operator) \cite{Ku} on $\g$ with respect to $(W,\rho)$ if
\[
[Ku,Kv]_{\g}
=
K\bigl(\rho(Ku)v-\rho(Kv)u\bigr),
\qquad \forall\,u,v\in W.
\]

We now adapt this notion to the equivariant Nijenhuis setting.
\begin{defi}
Let $(\mathfrak g,[\cdot,\cdot]_{\mathfrak g},E)$ be an \textup{ENL} algebra and
$(W;T,\rho)$ be an \textup{EN}-representation of $(\mathfrak g,E)$.
A linear map $K:W\to\mathfrak g$ is called an
\emph{\textup{EN}-relative Rota--Baxter operator}
on $(\mathfrak g,[\cdot,\cdot]_{\mathfrak g},E)$ with respect to
$(W;T,\rho)$ if $K$ is a relative Rota--Baxter operator on $\mathfrak g$
and $K$ is compatible with the equivariant Nijenhuis operators, in the sense that
\[
E\circ K = K\circ T.
\]
If $(W;T,\rho)=(\mathfrak g;E,\operatorname{ad})$, then $K$ is called an
\emph{\textup{EN} Rota--Baxter operator of weight $0$} on
$(\mathfrak g,[\cdot,\cdot]_{\mathfrak g},E)$.
\end{defi}
\begin{defi}
	Let $(\g,[\cdot,\cdot]_{\g},E)$ be an \ENL algebra and
	$(W;T,\rho)$ be an \ENE-representation of $(\g,E)$.
	A linear map $K:W\to\g$ is called an \emph{\ENE-relative Rota--Baxter operator}
	on $(\g,[\cdot,\cdot]_{\g},E)$ with respect to $(W;T,\rho)$ if $K$ is a relative Rota--Baxter operator on $\g$ and
   $K$ is compatible with the equivariant Nijenhuis operators,
		in the sense that
\[
		E\circ K=K\circ T.
		\]
	If $(W;T,\rho)=(\g;E,\ad)$, then $K$ is called an
	\emph{\EN Rota-Baxter operator of weight $0$} on
	$(\g,[\cdot,\cdot]_{\g},E)$.
	\end{defi}
\begin{rmk}
	The above definition extends the classical notion of
	Rota-Baxter operators recalled in Section~\ref{sec:bia2}.
	\ENE-relative Rota-Baxter operators may be viewed as
	representation-level analogues of Rota-Baxter operators, where the
	adjoint \ENE-representation is replaced by an arbitrary \ENE-representation,
	together with compatibility with the equivariant Nijenhuis operator.
	In particular, such operators provide a systematic mechanism for
	transporting \ENL structures from a given \ENL algebra to its
	representation spaces.
\end{rmk}

\begin{pro}\label{descendant-AveLie}
	Let $(\g,[\cdot,\cdot]_{\g},E)$ be an \ENL algebra and
	$K:W\to\g$ be an \ENE-relative Rota-Baxter operator with respect to the
	\ENE-representation $(W;T,\rho)$.
	Define a bilinear operation on $W$ by
	\begin{equation}\label{relative-Rotabaxter-bracket}
		[u,v]_K := \rho(Ku)v-\rho(Kv)u,
		\qquad \forall\,u,v\in W.
	\end{equation}
	Then $(W,[\cdot,\cdot]_K,T)$ is an \ENL algebra, called the
	descendant \ENL algebra associated with $K$.
\end{pro}

\begin{proof}
	By the classical theory of relative Rota--Baxter operators,
	$(W,[\cdot,\cdot]_K)$ is a Lie algebra.
	It remains to verify that the operator $T:W\to W$ is equivariant with
	respect to the bracket $[\cdot,\cdot]_K$.
	
	Using the equivariance of the representation $(W;T,\rho)$ and the
	compatibility condition $E\circ K=K\circ T$, we compute, for all
	$u,v\in W$,
	\begin{align*}
		T[u,v]_K
		&=T\bigl(\rho(Ku)v\bigr)-T\bigl(\rho(Kv)u\bigr)\\
		&=\rho(E(Ku))v-\rho(E(Kv))u\\
		&=\rho(K(Tu))v-\rho(K(Tv))u\\
		&=[Tu,v]_K.
	\end{align*}
	By symmetry of the above argument, one also has
	$[Tu,v]_K=[u,Tv]_K$.
	Hence $T$ is equivariant with respect to $[\cdot,\cdot]_K$, and
	$(W,[\cdot,\cdot]_K,T)$ is an \ENL algebra.
\end{proof}
\begin{cor}\label{cor:K-hom}
	Let $(\g,[\cdot,\cdot]_{\g},E)$ be an \ENL algebra and
	$K:W\to\g$ be an \ENE-relative Rota-Baxter  operator with respect to an
	\ENE-representation $(W;T,\rho)$.
	Then $K$ is an \ENE-homomorphism of \ENL algebras from the
	descendant \ENL algebra $(W,[\cdot,\cdot]_K,T)$ to
	$(\g,[\cdot,\cdot]_{\g},E)$.
\end{cor}

It is a classical result that a relative Rota-Baxter operator
$K:W\to\g$ induces a matched pair of Lie algebras, \cite{Ku},
\[
\bigl((\g,[\cdot,\cdot]_{\g}),\,(W,[\cdot,\cdot]_K),\,\rho,\,\mu\bigr),
\]
where the Lie bracket on $W$ is given by
\eqref{relative-Rotabaxter-bracket}, and the action
$\mu:W\to\gl(\g)$ is defined by
\begin{equation}\label{mu-matched}
	\mu(u)x
	=
	K(\rho(x)u)-[x,Ku]_{\g},
	\qquad \forall\,u\in W,\ x\in\g .
\end{equation}

The \ENE-compatibility conditions imposed on $K$ and on the
representation $(W;T,\rho)$ ensure that this matched pair
construction lifts naturally to the equivariant Nijenhuis setting,
yielding a matched pair of \ENL algebras.

\begin{pro}\label{construction-matched-pair}
	Let $(\g,[\cdot,\cdot]_{\g},E)$ be an \ENL algebra and
	$K:W\to\g$ be an \ENE-relative Rota-Baxter operator with respect to the \ENE-representation $(W;T,\rho)$.
	Let $(W,[\cdot,\cdot]_K,T)$ be the associated descendant \ENL algebra.
	Then $((\g,E),(W,T);\rho,\mu)$ forms a matched pair of \ENL algebras, where $\mu$ is defined by \eqref{mu-matched}.
\end{pro}

\begin{proof}
	By the classical theory of relative Rota--Baxter operators, the quadruple
	\[
	\bigl((\g,[\cdot,\cdot]_{\g}),\,(W,[\cdot,\cdot]_K),\,\rho,\,\mu\bigr)
	\]
	is a matched pair of Lie algebras. It remains to check compatibility with the \EN operators.
	Since $(W;T,\rho)$ is an \ENE-representation of $(\g,E)$ and $K$ is \ENE-compatible, i.e., $E\circ K=K\circ T$, a direct computation using \eqref{mu-matched} evaluated at $Ex$ shows that for all $x\in\g$ and $u\in W$,
	\[
	\mu(u)(Ex) = \mu(Tu)x = E(\mu(u)x).
	\]
	This verifies that $(\g;E,\mu)$ is an \ENE-representation of the descendant \ENL algebra $(W,[\cdot,\cdot]_K,T)$.
	Hence $((\g,E),(W,T);\rho,\mu)$ is indeed a matched pair of \ENL algebras.
\end{proof}
\medskip

Recall that for a skew-symmetric tensor $r\in\g\otimes\g$, the map
$r_+:\g^*\to\g$ is a relative Rota-Baxter operator on the Lie algebra $\g$ with
respect to the coadjoint representation $(\g^*,\ad^*)$ if and only if $r$
satisfies the classical Yang-Baxter equation $\lcf r,r\rcf=0$. In the \ENL setting,
\ENE-relative Rota-Baxter operators provide an alternative description of solutions
to the classical Yang-Baxter equation, as we now explain.

\begin{pro}\label{pro:rrr}
	Let $(\g,[\cdot,\cdot]_{\g},E)$ be an \ENL algebra and $r\in\g\otimes\g$ be
	skew-symmetric. Define the linear map $r_+:\g^*\to\g$ by
	\[
	\langle \eta, r_+(\xi)\rangle := \langle \xi\otimes\eta, r\rangle,
	\qquad \forall\,\xi,\eta\in\g^*.
	\]
	Then $r_+$ is an \ENE-relative Rota-Baxter operator on
	$(\g,[\cdot,\cdot]_{\g},E)$ with respect to the coadjoint \ENE-representation
	$(\g^*;E^*,\ad^*)$ if and only if $r$ is an \EN $r$-matrix.
\end{pro}

\begin{proof}
	If $r$ is skew-symmetric and satisfies the classical Yang-Baxter equation,
	then $r_+:\g^*\to\g$ is a relative Rota-Baxter operator on the Lie algebra
	$\g$ with respect to the coadjoint representation $(\g^*,\ad^*)$, i.e.,
	\[
	[r_+(\xi),r_+(\eta)]_{\g}
	=
	r_+\bigl(\ad^*_{r_+(\xi)}\eta-\ad^*_{r_+(\eta)}\xi\bigr),
	\qquad \forall\,\xi,\eta\in\g^*.
	\]
	It remains to verify the \EN compatibility. The map $r_+$ is \ENE-compatible if and only if
	\[
	E\circ r_+ = r_+\circ E^*.
	\]
	Pairing both sides with $\eta\in\g^*$ and using the definition of $r_+$ gives
	\[
	\langle \eta, E(r_+(\xi))\rangle
	=
	\langle \xi\otimes\eta, (E\otimes\Id)r\rangle,
	\qquad
	\langle \eta, r_+(E^*\xi)\rangle
	=
	\langle \xi\otimes\eta, (\Id\otimes E)r\rangle.
	\]
	Hence $E\circ r_+=r_+\circ E^*$ holds if and only if
	\[
	(E\otimes\Id)r=(\Id\otimes E)r,
	\]
	or equivalently,
	$
	(\Id\otimes E - E\otimes \Id)(r)=0,
	$
	which is exactly the additional condition defining an \EN $r$-matrix. Combining
	this with the classical equivalence completes the proof.
\end{proof}
\medskip

In the presence of an invariant nondegenerate bilinear form, relative Rota-Baxter operators
with respect to the coadjoint representation reduce to ordinary Rota-Baxter operators of weight zero.

\begin{pro}\label{pro:22}
	Let $(\g,[\cdot,\cdot]_{\g},E,S)$ be a quadratic \ENL algebra and
	$K:\g^*\to\g$ be a linear map.
	Then $K$ is an \ENE-relative Rota-Baxter operator on
	$(\g,[\cdot,\cdot]_{\g},E)$ with respect to the coadjoint \ENE-representation
	$(\g^*;E^*,\ad^*)$ if and only if the map
	\[
	B := K \circ S^{\sharp} : \g \to \g
	\]
	is a Rota-Baxter operator of weight $0$ on the \ENL algebra
	$(\g,[\cdot,\cdot]_{\g},E)$.
\end{pro}

\begin{proof}
	For all $x,y\in\g$, since $S$ is invariant and nondegenerate, by Theorem~\ref{quadratic-Eder} we have
	\[
	S^{\sharp}([x,y]_{\g}) = \ad_x^*(S^{\sharp}y) - \ad_y^*(S^{\sharp}x).
	\]
	Using this, we compute
	\begin{align*}
		B([Bx,y]_{\g} + [x,By]_{\g})
		= K\bigl(S^{\sharp}([Bx,y]_{\g} + [x,By]_{\g})\bigr) = K\bigl(\ad^*_{Bx} S^{\sharp}y - \ad^*_{By} S^{\sharp}x\bigr).
	\end{align*}
	Thus, the Rota-Baxter identity for $B$ is equivalent to the relative Rota-Baxter identity for $K$ with respect to the coadjoint representation.
		For \EN compatibility, using \eqref{quad-intertwine-E}, we have
	\[
	E \circ B = B \circ E
	\quad \Longleftrightarrow \quad
	(E\circ K) \circ S^{\sharp} = (K \circ E^*) \circ S^{\sharp}.
	\]
	Since $S^{\sharp}$ is an isomorphism, this holds if and only if
	\[
	E \circ K = K \circ E^*.
	\]
	Hence, $K$ is an \ENE-relative Rota-Baxter operator if and only if
	$B = K \circ S^{\sharp}$ is a Rota-Baxter  operator of weight $0$ on
	$(\g,[\cdot,\cdot]_{\g},E)$. Equivalently, $(\g,[\cdot,\cdot]_{\g},E,S,B)$
	defines a quadratic \ENL Rota-Baxter  algebra.
\end{proof}

As an immediate consequence of Proposition~\ref{pro:22} and the identification of \ENE-relative Rota-Baxter operators in the quadratic setting, we obtain the following.

\begin{cor}\label{cor:rplus-quadratic-RB}
	Let $(\g,[\cdot,\cdot]_{\g},E,S)$ be a quadratic \ENL algebra and
	$r\in\g\otimes\g$ be skew-symmetric. Then the map
	\[
	B := r_+ \circ S^{\sharp} : \g \to \g
	\]
	defines a Rota-Baxter operator of weight $0$ on $(\g,[\cdot,\cdot]_{\g},E)$
	if and only if $r$ is an \EN $r$-matrix.
	In particular, this provides an explicit way to construct \EN Rota-Baxter operators
	from quadratic \EN $r$-matrices.
\end{cor}
\medskip

It is a classical result due to Bai \cite{Bai} that relative Rota-Baxter
operators give rise to solutions of the classical Yang-Baxter equation
on semidirect product Lie algebras. More precisely, let $(\g,[\cdot,\cdot]_\g)$
be a Lie algebra, let $(W;\rho)$ be a representation of $\g$ and $K: W \to \g$ be a linear map. Using the canonical identification of a linear map $K:W\to\g$ with an element
\[
\bar K \in (\g\ltimes_{\rho^*}W^*)\otimes(\g\ltimes_{\rho^*}W^*),
\]
we define the skew--symmetric tensor
$
r_K := \bar K - \sigma(\bar K),
$
where $\sigma$ denotes the flip map. Then $r_K$ is a skew-symmetric solution
of the classical Yang-Baxter equation in the Lie algebra $\g \ltimes_{\rho^*} W^*$
if and only if $K$ is a relative Rota--Baxter  operator on $\g$ with respect
to $(W;\rho)$.

We now extend the classical correspondence between relative Rota-Baxter operators and solutions of the classical Yang-Baxter equation to the \ENL setting.

Let $(\g,[\cdot,\cdot]_{\g},E)$ be an \ENL algebra and let $(W;T,\rho)$ be an \ENE-representation of $\g$. Then the semidirect product Lie algebra $\g\ltimes_{\rho^*} W^*$ inherits a natural \ENL structure given by the operator $E+T^*$.

\begin{thm}\label{skew-solution-Equivariant-Nijenhuis}
	Let $K:W\to\g$ be an \ENE-relative Rota-Baxter operator on the \ENL algebra $(\g,[\cdot,\cdot]_{\g},E)$ with respect to the \ENE-representation $(W;T,\rho)$. Then
	\[
	r_K := \bar K - \sigma(\bar K)
	\]
	is a skew-symmetric solution of the classical Yang-Baxter equation in the semidirect product \ENL algebra $(\g\ltimes_{\rho^*}W^*, E+T^*)$.
\end{thm}

\begin{proof}
	By Bai's classical result \cite{Bai}, $r_K$ satisfies the classical Yang-Baxter equation in the Lie algebra $\g\ltimes_{\rho^*}W^*$ if and only if $K$ is a relative Rota-Baxter  operator on $\g$ with respect to the representation $(W,\rho)$.
	It remains to verify compatibility with the \EN operator $E+T^*$. For all $\xi,\eta\in \g^*$ and $u,v \in W$, we compute:
	\begin{align*}
		 \bigl((E+T^*)\otimes \Id - \Id \otimes (E+T^*)\bigr)\bar K(\xi+u, \eta+v)
		&  = \bar K(E^*\xi + Tu, \eta + v) - \bar K(\xi + u, E^*\eta + Tv) \\
		&  = \langle K(Tu), \eta \rangle - \langle K u, E^* \eta \rangle \\
		&   = \langle (K \circ T - E \circ K) u, \eta \rangle = 0,
	\end{align*}
	where the last equality uses the \ENE-compatibility condition $E \circ K = K \circ T$.
	Since the flip map $\sigma$ anti-commutes with $(E+T^*)\otimes \Id - \Id \otimes (E+T^*)$, the same calculation applies to $\sigma(\bar K)$, yielding
	\[
	\bigl((E+T^*) \otimes \Id - \Id \otimes (E+T^*)\bigr) r_K = 0.
	\]
	Hence $r_K$ satisfies the \EN compatibility condition, and therefore is a skew-symmetric solution of the classical Yang-Baxter equation in the \ENL algebra $(\g\ltimes_{\rho^*}W^*, E+T^*)$.
\end{proof}
\section{Equivariant Nijenhuis operators on pre-Lie algebras: pre-\ENL algebras}
\label{sec.7}

Pre-Lie algebras refine Lie algebra structures by encoding additional, ordered information through a nonassociative product whose commutator recovers a Lie bracket. In this section, we introduce \emph{equivariant Nijenhuis pre-Lie algebras} and explain how they naturally fit into the \ENL framework developed above.

\medskip

Let $(\g,\{\cdot,\cdot\})$ be a pre-Lie algebra, i.e., a vector space $\g$ equipped with a bilinear product
$
\{\cdot,\cdot\}:\g\otimes \g\to \g
$
whose associator
\[
(x,y,z):=\{\{x,y\},z\}-\{x,\{y,z\}\}
\]
is symmetric in the first two arguments. Equivalently, for all $x,y,z\in \g$,
\begin{equation}\label{pre-Lie-defi}
	\{\{x,y\},z\}-\{x,\{y,z\}\}=\{\{y,x\},z\}-\{y,\{x,z\}\}.
\end{equation}
The induced Lie bracket is given by
\[
[x,y]:=\{x,y\}-\{y,x\}.
\]
Nijenhuis operators on pre-Lie algebras arise in deformation theory; see, e.g., \cite{Burde,Gerstenhaber}.

\begin{defi}
	A \emph{Nijenhuis operator} on a pre-Lie algebra $(\g,\{\cdot,\cdot\})$ is a linear map $N:\g\to\g$ satisfying
	\begin{equation}\label{eq:preLie-Nijenhuis}
		\{Nx,Ny\} = N\big(\{Nx,y\} + \{x,Ny\} - N\{x,y\}\big), \qquad \forall x,y\in\g.
	\end{equation}
	A pre-Lie algebra equipped with such an operator is called a \emph{pre-\NL algebra}.
\end{defi}
Given $N$, one defines the $N$-deformed product
\[
\{x,y\}_N := \{Nx,y\} + \{x,Ny\} - N\{x,y\}, \qquad \forall x,y\in\g.
\]
If $N$ satisfies \eqref{eq:preLie-Nijenhuis}, then $(\g,\{\cdot,\cdot\}_N)$ is again a pre-Lie algebra. Moreover, $N$ is a Nijenhuis operator for the induced Lie algebra $(\g,[\cdot,\cdot])$, and the Lie bracket associated with $\{\cdot,\cdot\}_N$ coincides with the Lie-Nijenhuis deformation:
\[
[x,y]_N := [Nx,y] + [x,Ny] - N[x,y].
\]
In general, the converse does not hold: a Nijenhuis operator on the induced Lie algebra need not be a Nijenhuis operator for the underlying pre-Lie structure. The same phenomenon arises for equivariance conditions. Unlike the Lie algebra case, where equivariance is naturally formulated with respect to the adjoint representation, a pre-Lie algebra lacks a canonical adjoint pre-Lie action. Its only intrinsic symmetry is the adjoint representation of the induced Lie algebra. Consequently, one must exercise care when formulating equivariance conditions for operators on pre-Lie algebras, as discussed below.
\begin{lem}\label{lem:strong-preLie-equiv-implies-Lie-equiv}
	Let $(\g,\{\cdot,\cdot\})$ be a pre-Lie algebra with induced Lie bracket $[x,y]=\{x,y\}-\{y,x\}$, and $E:\g\to\g$ be a linear map.
	\begin{enumerate}
		\item If $E$ satisfies the compatibility condition
		\begin{equation}\label{eq:strong-preLie-equiv}
			E\{x,y\}=\{Ex,y\}=\{x,Ey\}, \qquad \forall x,y\in\g,
		\end{equation}
		then $E$ is equivariant for the induced Lie algebra:
		\begin{equation}\label{eq:Lie-equiv-from-strong}
			E([x,y])=[Ex,y]=[x,Ey], \qquad \forall x,y\in\g.
		\end{equation}

		\item The converse is false in general: the Lie-equivariance condition \eqref{eq:Lie-equiv-from-strong}
		does not imply \eqref{eq:strong-preLie-equiv} without additional assumptions.
	\end{enumerate}
\end{lem}
\begin{proof}
	{\rm (i)} Using $[x,y]=\{x,y\}-\{y,x\}$ and \eqref{eq:strong-preLie-equiv}, we compute
\[
E([x,y])
=E\{x,y\}-E\{y,x\}
=\{x,Ey\}-\{y,Ex\}.
\]
On the other hand,
\[
[x,Ey]=\{x,Ey\}-\{Ey,x\}.
\]
Applying \eqref{eq:strong-preLie-equiv} to $(y,x)$ gives
$\{y,Ex\}=E\{y,x\}=\{Ey,x\}$, hence $E([x,y])=[x,Ey]$.
Thus \eqref{eq:Lie-equiv-from-strong} holds.
	
	{\rm (ii)}  Condition \eqref{eq:Lie-equiv-from-strong} only constrains the skew-symmetrization of the pre-Lie product. It does not determine $E\{x,y\}$ from $\{x,Ey\}$ nor $E\{y,x\}$ from $\{Ey,x\}$. Hence, the condition \eqref{eq:strong-preLie-equiv} does not follow in general.
\end{proof}
In the following, we distinguish two notions of \EN structures on pre-Lie algebras.

\begin{defi}\label{def:pre-ENL}
	Let $(\g,\{\cdot,\cdot\})$ be a pre-Lie algebra, and let
	$[x,y]:=\{x,y\}-\{y,x\}$ denote the induced Lie bracket.
	Let $E:\g\to\g$ be a linear map. We say that $(\g,\{\cdot,\cdot\},E)$ is a \emph{pre-\ENL algebra} if $E$
		satisfies the compatibility condition
		\[
		E\{x,y\} = \{Ex,y\} = \{x,Ey\}, \qquad \forall x,y \in \g.
		\]	
\end{defi}
By Lemma~\ref{lem:strong-preLie-equiv-implies-Lie-equiv}, every pre-\ENL algebra induces an \ENL structure on the associated Lie algebra. This notion is relevant for the \ENL constructions developed in the previous sections.


\begin{pro}\label{Equivariant-Nijenhuis-Pre-conclusion3}
	Let $(\g,\{\cdot,\cdot\},E)$ be a pre-\ENL algebra. Then:
	\begin{itemize}
		\item[{\rm (i)}] $(\g,[\cdot,\cdot]_{\g},E)$ is an \ENL algebra, where
		\[
		[x,y]_{\g}:=\{x,y\}-\{y,x\}, \qquad \forall x,y\in\g.
		\]
		We call $(\g,[\cdot,\cdot]_{\g},E)$ the \emph{sub-adjacent \ENL algebra} of $(\g,\{\cdot,\cdot\},E)$.
		
		\item[{\rm (ii)}] Define $L:\g\to\gl(\g)$ by $L(x)y=\{x,y\}$. Then $(\g;E,L)$ is a representation of the \ENL algebra $(\g,[\cdot,\cdot]_{\g},E)$.
		
		\item[{\rm (iii)}] The identity map $\Id:\g\to\g$ is a relative Rota-Baxter operator on $(\g,[\cdot,\cdot]_{\g},E)$ with respect to the representation $(\g;E,L)$.
	\end{itemize}
\end{pro}

\begin{proof}
	{\rm (i)}
	By Lemma~\ref{lem:strong-preLie-equiv-implies-Lie-equiv}, the compatibility
	condition implies equivariance with respect to the induced Lie bracket.
	Hence $(\g,[\cdot,\cdot]_{\g},E)$ is an \ENL algebra.

	{\rm (ii)} For all $x,y,z\in\g$, the pre-Lie identity \eqref{pre-Lie-defi} implies
	$L([x,y]_{\g})=[L(x),L(y)]$. Moreover, we have
	\[
	L(Ex)=E\circ L(x)=L(x)\circ E.
	\]
	Thus $(\g;E,L)$ is a representation of $(\g,[\cdot,\cdot]_{\g},E)$.
	
{\rm (iii)} This follows directly from the definition of relative Rota-Baxter
	operator and the fact that $[x,y]_{\g}=\{x,y\}-\{y,x\}$.
	\end{proof}
\begin{pro}\label{RRB-Equivariant-Nijenhuis-Pre-Lie}
	Let $K:W\rightarrow\g$ be a relative Rota-Baxter operator on an \ENL algebra
	$(\g,[\cdot,\cdot]_{\g},E)$ with respect to an \ENE-representation $(W;T,\rho)$.
	Then $(W,\{\cdot,\cdot\}_{K},T)$ is a pre-\ENL algebra, where
	\[
	\{u,v\}_{K}=\rho(Ku)v,\qquad \forall u,v\in W.
	\]
\end{pro}
\begin{proof}
	It is well known that if $K$ is a relative Rota-Baxter operator on a Lie algebra
	$(\g,[\cdot,\cdot]_{\g})$ with respect to a representation $(W;\rho)$, then
	$(W,\{\cdot,\cdot\}_{K})$ is a pre-Lie algebra.
		Using $E\circ K=K\circ T$ and the fact that $(W;T,\rho)$ is an \ENE-representation,
	for all $u,v\in W$ we compute
	\[
	T\{u,v\}_{K}
	=
	T(\rho(Ku)v)
	=
	\rho(E(Ku))v
	=
	\rho(K(Tu))v
	=
	\{Tu,v\}_{K}.
	\]
	Similarly, one has $T\{u,v\}_{K}=\{u,Tv\}_{K}$.
	Hence $T$ satisfies the equivariance condition with respect to the
	pre-Lie product $\{\cdot,\cdot\}_{K}$, and therefore
	$(W,\{\cdot,\cdot\}_{K},T)$ is a pre-\ENL algebra.
\end{proof}
\begin{rmk}
	With the notation above, the Lie bracket induced by the pre-Lie product
	$\{\cdot,\cdot\}_{K}$ satisfies
	\[
	[u,v]_K=\{u,v\}_{K}-\{v,u\}_{K}.
	\]
	Hence the \ENL algebra $(W,[\cdot,\cdot]_K,T)$ constructed in
	Proposition~\ref{descendant-AveLie} is precisely the subadjacent \ENL Lie algebra
	of the pre-\ENL algebra $(W,\{\cdot,\cdot\}_{K},T)$.
\end{rmk}
	
	\begin{pro}\label{prop:preENL-iff-invertible-RRB}
		Let $(\g,[\cdot,\cdot]_{\g},E)$ be an \ENL algebra. Then there exists a
		 pre-\ENL algebra $(\g,\{\cdot,\cdot\},E)$ whose subadjacent \ENL Lie
		algebra is $(\g,[\cdot,\cdot]_{\g},E)$ if and only if there exist an \ENE-representation
		$(W;T,\rho)$ of $(\g,[\cdot,\cdot]_{\g},E)$ and an invertible \ENE-relative
		Rota-Baxter  operator $K:W\to\g$ with respect to $(W;T,\rho)$.
		In this case, the pre-Lie product on $\g$ is given by
		\[
		\{x,y\}_{K}:=K\big(\rho(x)K^{-1}y\big),\qquad \forall x,y\in\g,
		\]
		and $(\g,\{\cdot,\cdot\}_{K},E)$ is a pre-\ENL algebra.
	\end{pro}

	\begin{proof}
		Assume that $K:W\to\g$ is an invertible \ENE-relative Rota-Baxter  operator on
		$(\g,[\cdot,\cdot]_{\g},E)$ with respect to an \ENE-representation $(W;T,\rho)$.
		By Proposition~\ref{RRB-Equivariant-Nijenhuis-Pre-Lie}, $(W,\{\cdot,\cdot\}_{K},T)$
		is a pre-\ENL algebra, where $\{u,v\}_{K}=\rho(Ku)v$.
		Transporting the pre-Lie product to $\g$ by $K$, we obtain, for all $x,y\in\g$,
		\[
		\{x,y\}_{K}:=K\big\{K^{-1}x,K^{-1}y\big\}_{K}
		=K\big(\rho(x)K^{-1}y\big),
		\]
		which defines a pre-Lie algebra structure on $\g$.
				Moreover, using $E\circ K=K\circ T$ and the \ENE-representation property
		$T\circ\rho(x)=\rho(Ex)\circ T$, we have
		\[
		E\{x,y\}_{K}
		=E\Big(K(\rho(x)K^{-1}y)\Big)
		=K\Big(T(\rho(x)K^{-1}y)\Big)
		=K\Big(\rho(Ex)K^{-1}y\Big)
		=\{Ex,y\}_{K}.
		\]
		Similarly, since $T(K^{-1}y)=K^{-1}(Ey)$, we obtain
		\[
		E\{x,y\}_{K}
		=K\Big(T(\rho(x)K^{-1}y)\Big)
		=K\Big(\rho(x)T(K^{-1}y)\Big)
		=K\Big(\rho(x)K^{-1}(Ey)\Big)
		=\{x,Ey\}_{K}.
		\]
		Hence $(\g,\{\cdot,\cdot\}_{K},E)$ is a pre-\ENL algebra, and its
		subadjacent \ENL Lie algebra is $(\g,[\cdot,\cdot]_{\g},E)$.
		
		Conversely, let $(\g,\{\cdot,\cdot\},E)$ be a pre-\ENL algebra whose
		subadjacent \ENL algebra is $(\g,[\cdot,\cdot]_{\g},E)$.
		By Proposition~\ref{Equivariant-Nijenhuis-Pre-conclusion3}, $(\g;E,L)$ is an
		\ENE-representation of $(\g,[\cdot,\cdot]_{\g},E)$, where $L(x)y=\{x,y\}$.
		Moreover, the identity map $\Id:\g\to\g$ is an \ENE-relative Rota-Baxter  operator
		on $(\g,[\cdot,\cdot]_{\g},E)$ with respect to $(\g;E,L)$.
		This completes the proof.
	\end{proof}
\begin{thm}\label{thm:CYBE-from-preENL}
	Let $(\g,\{\cdot,\cdot\},E)$ be a pre-\ENL algebra and let $L:\g\to\gl(\g)$ be
	the left multiplication representation, $L(x)y=\{x,y\}$.
	Let $\{e_1,\dots,e_n\}$ be a basis of $\g$ and $\{e_1^*,\dots,e_n^*\}$ its dual basis.
	Define
	\[
	r:=\sum_{i=1}^n\big(e_i\otimes e_i^*-e_i^*\otimes e_i\big)\in
	(\g\ltimes_{L^*}\g^*)\otimes(\g\ltimes_{L^*}\g^*).
	\]
	Then $r$ is an \EN $r$-matrix in the semidirect product \ENL algebra
	$(\g\ltimes_{L^*}\g^*,\,E+E^*)$.
\end{thm}

\begin{proof}
	Let $[\cdot,\cdot]_{\g}$ be the induced Lie bracket on $\g$,
	$[x,y]_{\g}:=\{x,y\}-\{y,x\}$.
	By Proposition~\ref{Equivariant-Nijenhuis-Pre-conclusion3}, $(\g;E,L)$ is an
	\ENE-representation of the \ENL algebra $(\g,[\cdot,\cdot]_{\g},E)$, and the identity
	map $\Id:\g\to\g$ is an \ENE-relative Rota-Baxter  operator on
	$(\g,[\cdot,\cdot]_{\g},E)$ with respect to $(\g;E,L)$.
	Therefore, by Theorem~\ref{skew-solution-Equivariant-Nijenhuis}, the element
	$r=\sum_{i=1}^n(e_i\otimes e_i^*-e_i^*\otimes e_i)$ is a skew-symmetric solution of
	the classical Yang-Baxter equation in the \ENL algebra
	$(\g\ltimes_{L^*}\g^*,E+E^*)$.
\end{proof}
\begin{rmk}
	This construction shows that pre-\ENL algebras provide a natural
	pre-Lie refinement of \ENL algebras, compatible with relative
	Rota-Baxter operators and classical $r$-matrix constructions.
\end{rmk}

\section{Concluding Remarks}

In this work, we have introduced the framework of equivariant Nijenhuis Lie (\ENL) algebras and established operator-equivariant analogues of the classical Lie-theoretic correspondences between matched pairs, Manin triples, and Drinfel'd doubles. This setting incorporates equivariant Rota--Baxter operators and provides a systematic mechanism for constructing classical $\EN$ $r$-matrices and the associated \ENL bialgebra structures.

\smallskip

Several natural directions emerge from this framework. Nijenhuis-type operators play a central role in integrable and dynamical systems, where they generate hierarchies of compatible brackets and commuting flows. The \ENL framework extends this paradigm to a symmetric and equivariant setting compatible with double constructions and quadratic structures. In particular, it offers a structured environment in which compatibility, duality, and $r$-matrix constructions coexist coherently.
Further investigation may include extensions to deformation and quantum settings, where operator-equivariant structures could interact with quantum group constructions and operator-theoretic approaches to quantization. Moreover, the interplay between \ENL bialgebras, equivariant relative Rota--Baxter operators, and classical $r$-matrices suggests new families of solutions to the classical Yang--Baxter equation, with potential applications in geometry and mathematical physics.

We expect that the \ENL framework provides a flexible structural bridge between classical Lie theory, operator-based approaches to integrability, and emerging perspectives on quantization.

	\vspace{2mm}
	\noindent
	{\bf Acknowledgements. }  This research is supported by  NSF of Jilin Province (20260101013JJ) and  NSFC (12301034, 12461004, 12471060, W2412041).  Z.~Ravanpak gratefully acknowledges the hospitality of the research group at the Department of Mathematics, Jilin University, China, in September 2024, as well as the stimulating discussions following an invited talk, which helped shape the ideas of this work. She also acknowledges a research stay at the Erwin Schrödinger International Institute for 
Mathematics and Physics (ESI), University of Vienna, in October–December 2024 and in February 2025, where part of this 
work was carried out.


\begin{thebibliography}{ab}	
	\bibitem{Aguiar} M. Aguiar, Pre-Poisson algebras, \emph{Lett. Math. Phys.} {\bf54} (2000), 263-277.
	
	\bibitem{Andrada} A. Andrada and S. Salamon, Complex product structures on Lie algebras, \emph{Forum Math.} {\bf17} (2005), no. 2, 261-295.
	
	
	\bibitem{Bai} C. Bai, A unified algebraic approach to the classical Yang-Baxter equation, \emph{J. Phys. A: Math. Theor.} {\bf40} (2007), 11073-11082.
	
	\bibitem{BGN2010} C. Bai, L. Guo and X. Ni, Nonabelian generalized Lax pairs, the classical Yang--Baxter equation and PostLie algebras, \emph{Comm. Math. Phys.} {\bf297} (2010), 553-596.

		\bibitem{BaMaRa}
	A. Ballesteros, J. C. Marrero and Z. Ravanpak,  Poisson-Lie groups, bi-Hamiltonian systems and integrable deformations,  \emph{J. Phys. A: Math. Theor.} {\bf 50} (2017) 145204.
	
	\bibitem{BD} A. A. Belavin and V. G. Drinfel'd, Solutions of the classical Yang-Baxter equation for simple Lie algebras, \emph{Funct. Anal. Appl.} {\bf16} (1982), 159-180.
	
	
	\bibitem{Burde} D. Burde, Left-symmetric algebras, or pre-Lie algebras in geometry and physics, \emph{Central European J. Math.} {\bf4} (2006), no. 3, 323-357.
	
	\bibitem{CP} V. Chari and A. Pressley, \emph{A Guide to Quantum Groups}, Cambridge University Press, 1994.
	
	
	\bibitem{DamianouFernandes} P. A. Damianou and R. L. Fernandes, Integrable hierarchies and the modular class, \emph{Ann. Inst. Fourier} {\bf58} (2008), 107-137.
	
	\bibitem{Das-1} A. Das, Cohomology theory of Nijenhuis Lie algebras and (generic) Nijenhuis Lie bialgebras, arXiv:2502.16257.

	\bibitem{D} V. G. Drinfel'd, Hamiltonian structures on Lie groups, Lie bialgebras and the geometric meaning of the classical Yang Baxter equations, \emph{Soviet Math. Dokl.} {\bf27} (1983), 68-71.
	
	
	
	
	\bibitem{FokasFuchssteiner}
	A. S. Fokas and B. Fuchssteiner, Symplectic structures, their B\"acklund transformations and hereditary symmetries, \emph{Phys. D} {\bf4} (1981), 47-66.
	
	\bibitem{G3} M. Goncharov and V. Gubarev, Double Lie algebras of nonzero weight,
\emph{Adv. Math.} {\bf409} (2022), part B, Paper No. 108680, 30 pp.
	
	\bibitem{Goncharov} M. E. Goncharov and P. S. Kolesnikov, Simple finite-dimensional double algebras, \emph{J. Algebra} {\bf500} (2018), 425-438.
	
	\bibitem{Gerstenhaber} M. Gerstenhaber, The cohomology structure of an associative ring, \emph{Ann. Math.} {\bf78} (1963), 267-288.

	\bibitem{HaZo}
G. Haghighatdoost, Z. Ravanpak, and A. Rezaei-Aghdam,
 Some remarks on invariant Poisson quasi-Nijenhuis structures on Lie groups,	\emph{Int. J. Geom. Methods Mod. Phys.}  {\bf 16} (2019) 1950097.
	
	\bibitem{K} Y. Kosmann-Schwarzbach, Lie bialgebras, Poisson Lie groups and dressing transformations, In: Kosmann-Schwarzbach Y., Tamizhmani K.M., Grammaticos B. (eds) Integrability of Nonlinear Systems, \emph{Lecture Notes in Phys.} vol 638, Springer, Berlin, Heidelberg, 2004, 107-173.
	
	\bibitem{KoMa} Y. Kosmann-Schwarzbach and F. Magri, Poisson-Nijenhuis structures, \emph{Ann. Inst. Henri Poincare, Phys. Theor.} {\bf53} (1990), 35-81.
	
	\bibitem{KM} Y. Kosmann-Schwarzbach and F. Magri, Poisson-Lie groups and complete integrability I: Drinfeld bialgebras, dual extensions and their canonical representations, \emph{Ann. Inst. Henri Poincare} {\bf49} (1988), no. 4, 433-460.
	
	\bibitem{KosmannPN} Y. Kosmann-Schwarzbach, Poisson-Nijenhuis structures, in: \emph{Lie Algebroids and Related Topics in Differential Geometry}, Banach Center Publ. {\bf54} (2001), 1-23.
	
	\bibitem{Ku} B. A. Kupershmidt, What a classical $r$-matrix really is, \emph{J. Nonlinear Math. Phys.} {\bf6} (1999), 448-488.
	
	\bibitem{Lang-Sheng} H. Lang and Y. Sheng, Factorizable Lie bialgebras, quadratic Rota--Baxter  Lie algebras and Rota--Baxter  Lie bialgebras, \emph{Comm. Math. Phys.} {\bf397} (2023), no. 2, 763-791.
	
	\bibitem{Li-Ma} H. Li and T. Ma, Classical Yang-Baxter equations and Nijenhuis operators for Lie algebras, arXiv:2502.18717.
	
	
	\bibitem{Lu2} J-H. Lu and A. Weinstein, Poisson Lie groups, dressing transformations and Bruhat decompositions, \emph{J. Diff. Geom.} {\bf31} (1990), 501-526.
	
	\bibitem{Magri} F. Magri, A simple model of the integrable Hamiltonian equation, \emph{J. Math. Phys.} {\bf19} (1978), 1156-1162.
	
	\bibitem{MagriMorosiAnn} F. Magri and C. Morosi, A geometrical characterization of integrable Hamiltonian systems through the theory of Poisson-Nijenhuis manifolds, \emph{Ann. Inst. H. Poincare Phys. Theor.} {\bf40} (1984), 305-331.
	
	\bibitem{Majid} S. H. Majid, Matched pairs of Lie groups associated to solutions of the Yang--Baxter equations, \emph{Pacific J. Math.} {\bf141} (1990), no. 2, 311-332.
	
	\bibitem{Panasyuk2006} M. Panasyuk, Algebraic Nijenhuis operators and Kronecker Poisson pencils, \emph{Diff. Geom. Appl.} {\bf24} (2006), 482-491.
	
	\bibitem{RS} N. Reshetikhin and M. A. Semenov-Tian-Shansky, Quantum $R$-matrices and factorization problems, \emph{J. Geom. Phys.} {\bf5} (1988), 533-550.
	
	
	\bibitem{STS} M. A. Semenov-Tian-Shansky, What is a classical $r$-matrix? \emph{Funct. Anal. Appl.} {\bf17} (1983), 259-272.
	
	
	\bibitem{Takeuchi} M. Takeuchi, Matched pairs of groups and bismash products of Hopf algebras, \emph{Comm. Algebra} {\bf9} (1981), 841-882.
	
	
			\bibitem{RaReHa}
	Z. Ravanpak, A. Rezaei-Aghdam and Gh. Haghighatdoost, Invariant Poisson Nijenhuis structures on Lie groups and classification,
	\emph{Int. J. Geom. Methods Mod. Phys.} {\bf 15} (2018), 1850059.	
	
	\bibitem{Zohreh} Z. Ravanpak, NL bialgebras, \emph{Adv. Theor. Math. Phys.} {\bf29} (2025), no. 5,  1407-1445. 
	
	
	\bibitem{Vicedo2015} B. Vicedo, Deformed integrable $\sigma$-models, classical $R$-matrices and classical exchange algebra on Drinfel'd doubles, \emph{J. Phys. A: Math. Theor.} {\bf48} (2015), 355203.

\end{thebibliography}
\end{document}